\documentclass[11pt,reqno]{amsart}
\usepackage[colorlinks=true,
pdfstartview=FitV, linkcolor=cyan, citecolor=magenta,
urlcolor=blue]{hyperref}
\usepackage{amsmath,amsfonts,latexsym,amssymb}
\usepackage{mathrsfs}
\usepackage[latin1]{inputenc}
\usepackage[T1]{fontenc}
\usepackage{ae,aecompl}
\usepackage{braket}
\usepackage{comment}
\usepackage{color}
\usepackage{graphicx}
\usepackage{subfig}
\newtheorem{theorem}{Theorem}[section]
\newtheorem{lemma}[theorem]{Lemma}
\newtheorem{proposition}[theorem]{Proposition}
\newtheorem{corollary}[theorem]{Corollary}

\long\def\@savemarbox#1#2{\global\setbox#1\vtop{\hsize\marginparwidth 
  \@parboxrestore\tiny\raggedright #2}}
\newcommand{\Real}{\mathbb R}

\newcommand{\SL}{\mathsf{SL}(d,\Real)}

\def\eproof{$\Box$ \medskip}

\title[Topological restrictions on Anosov representations]{Topological restrictions on Anosov representations}
\author[Canary]{Richard Canary}
\address{University of Michigan, Ann Arbor, MI 41809 USA}
\author[Tsouvalas]{Konstantinos Tsouvalas}
\address{University of Michigan, Ann Arbor, MI 41809 USA}
\thanks{
Canary and Tsouvalas were partially supported by the grant DMS-1564362, from the National Science Foundation.}
\begin{document}
\begin{abstract}
We characterize groups admitting Anosov representations into $\mathsf{SL}(3,\mathbb R)$, projective Anosov
representations into $\mathsf{SL}(4,\mathbb R)$, and Borel Anosov representations into $\mathsf{SL}(4,\mathbb R)$.
More generally, we obtain bounds on the cohomological dimension of groups admitting $P_k$-Anosov representations into
$\mathsf{SL}(d,\mathbb R)$ and offer several characterizations of Benoist representations.
\end{abstract}

\maketitle

\section{Introduction}

In this paper, we investigate topological restrictions on hyperbolic groups admitting Anosov representations into
$\mathsf{SL}(d,\mathbb R)$. Anosov representations were introduced by
Labourie \cite{labourie-invent} in his work on Hitchin representations and, after 
further development, by Guichard-Wienhard \cite{guichard-wienhard}, Gu\'eritaud-Guichard-Kassel-Wienhard \cite{GGKW},
Kapovich-Leeb-Porti \cite{KLP1,KLP2} and others,
are widely recognized as the natural higher rank analogue of convex cocompact representations into rank one Lie groups.
Danciger-Gu\'eritaud-Kassel \cite{DGK0,DGK} and Zimmer \cite{zimmer} have recently shown that Anosov representations
may often be viewed as convex cocompact actions on convex domains in projective spaces.

Our study was motivated by the fact that most torsion-free hyperbolic groups which are known to admit Anosov representations are
isomorphic to convex cocompact subgroups of rank one Lie groups. In fact, the only known Anosov representations 
whose domain groups are  not isomorphic to convex
cocompact subgroups of rank one Lie groups are due to Benoist \cite{benoist-examples} and Kapovich \cite{kapovich} and, more recently, 
to Danciger, Gu\'eritaud, Kassel, Lee and Marquis \cite{DGK0,DGK2,lee-marquis}.
Kapovich constructs examples in $\mathsf{SL}(d,\mathbb R)$ for all $d\ge 5$.
We show that there are no ``new'' examples of projective Anosov representations  in dimension 3 or 4.

\begin{theorem}
\label{sl3rcase}
If $\Gamma$ is a torsion-free hyperbolic group and $\rho:\Gamma\to  \mathsf{SL}(3,\mathbb R)$ is an Anosov representation, 
then $\Gamma$  is either a free group or a surface group.
\end{theorem}

We will say that a representation $\rho:\Gamma\to \mathsf{SL}(d,\mathbb R)$ is $P_k$-Anosov if $k\le \frac{d}{2}$ and
$\rho$ is Anosov with respect to the parabolic group which is the stabilizer of a $k$-plane in $\mathbb R^d$. We will
refer to $P_1$-Anosov representations as projective Anosov representations. 

\begin{theorem}
\label{sl4rcase} 
If $\Gamma$ is a torsion-free hyperbolic group and $\rho:\Gamma\to  \mathsf{SL}(4,\mathbb R)$ is a projective Anosov representation, then $\Gamma$ 
is isomorphic to a convex cocompact subgroup of $\mathsf{PO}(3,1)$. In particular,  $\Gamma$ is the
fundamental group of a compact hyperbolizable 3-manifold.
\end{theorem}

Our most general result is that, with four explicit exceptions, if $\rho:\Gamma\to \mathsf{SL}(d,\mathbb R)$ is $P_k$-Anosov
then $\Gamma$ has  cohomological dimension at most $d-k$. We will see in the proof that each exception is related to one of
the four Hopf fibrations.

\begin{theorem}
\label{cohomdimd-k}
Suppose  $\Gamma$ is a torsion-free hyperbolic group and
$\rho:\Gamma\to  \mathsf{SL}(d,\mathbb R)$ is $P_k$-Anosov.
\begin{enumerate}
\item
If $(d,k)$ is not  $(2,1)$, $(4,2)$, $(8,4)$ or $(16,8)$, 
then $\Gamma$ has cohomological dimension at most $d-k$.
\item
If $(d,k)$ is $(2,1)$, $(4,2)$, $(8,4)$ or $(16,8)$, then $\Gamma$ has cohomological dimension at most $d-k+1$. 
Moreover, if $\Gamma$ has cohomological dimension $d-k+1$,
then $\partial \Gamma$ is homeomorphic to $S^{d-k}$ and, if $k\ne 1$, $\rho$ is not  projective Anosov.
\end{enumerate}
\end{theorem}

Benoist representations are one of the richest classes of examples of Anosov representations,
see, for example \cite{benoist-divisible1,benoist-survey}.
We recall that \hbox{$\rho:\Gamma\to  \mathsf{SL}(d,\mathbb R)$}  is a {\em Benoist representation}, 
if $\rho$ has finite kernel and $\rho(\Gamma)$ preserves and acts properly discontinuously and
cocompactly on a strictly convex domain in $\mathbb{P}(\mathbb R^d)$. Notice that, almost by definition,
$\rho(\Gamma)$ must have virtual cohomological dimension $d-1$ and recall that Benoist 
representations are projective Anosov (see \cite[Prop 6.1]{guichard-wienhard}).
Theorem \ref{cohomdimd-k} implies that they are {\em only} projective Anosov.

\begin{corollary}
\label{benoistnotpk}
If $\rho:\Gamma\to \mathsf{SL}(d,\mathbb R)$ is a Benoist representation and
$\frac{d}{2}\ge k\ge 2$, then $\rho$ is not $P_k$-Anosov. 
\end{corollary}

Conversely, we will see that, when $d\ge 4$, Benoist representations are characterized, among Anosov
representations, entirely by the cohomological dimension of their domain group.
 Notice that the Anosov representations of surface groups into $\mathsf{SL}(3,\mathbb R)$
studied by Barbot \cite{barbot} are counterexamples to 
the statement of Theorem \ref{cohomdimd} when $d=3$.

\begin{theorem}
\label{cohomdimd}
If $d\ge 4$, an Anosov representation $\rho:\Gamma\to\mathsf{SL}(d,\mathbb R)$ of a torsion-free hyperbolic group $\Gamma$
is a Benoist representation if and only if 
$\Gamma$ has cohomological dimension $d-1$.
\end{theorem}

Labourie \cite{labourie-invent} showed that Hitchin representations are irreducible and Borel Anosov, i.e. are $P_k$- Anosov for all $k$.
Andres Sambarino asked whether any torsion-free Borel Anosov subgroup of $\mathsf{SL}(d,\mathbb R)$ 
is either free or a surface group. Theorem \ref{sl3rcase} settles this question in the affirmative when $d=3$.
Here, we answer Sambarino's question when $d=4$. We know of no counterexamples in any dimension.

\begin{theorem}
\label{borel4}
If  $\Gamma$ is a torsion-free hyperbolic group and $\rho:\Gamma\to \mathsf{SL}(4,\mathbb R)$ is Borel Anosov,
then $\Gamma$ is either a surface group or a free group.
\end{theorem}

Hitchin representations are the only known Borel Anosov representations of surface groups when $d$ is even.
Hitchin representation are irreducible (see \cite[Lemma 10.1]{labourie-invent}), but
whenever $d$ is odd one may use Barbot's construction to produce reducible Borel Anosov representations of surface groups
into $\mathsf{SL}(d,\mathbb R)$.
In Proposition \ref{Borel4irreducible} we show that every representation of a surface group into $\mathsf{SL}(4,\mathbb R)$ is
irreducible.  One might hope that all Borel Anosov representations of surface groups into $\mathsf{SL}(4,\mathbb R)$
are Hitchin.
In Proposition \ref{hitchindontextend} we show that the restriction of a Borel Anosov representation to an infinite index surface
subgroup cannot be Hitchin, so this would provide another proof of Theorem \ref{borel4}.

We also extend Theorem \ref{cohomdimd} to replace the assumption that $\rho$ is Anosov with
the simpler assumption that $\rho$ admits an non-constant 
limit map into $\mathbb P(\mathbb R^d)$. 

\begin{theorem}
\label{benoistlimitmap}
Suppose that $d\ge 4$ and $\Gamma$ is a torsion-free hyperbolic group. 
A representation $\rho:\Gamma \rightarrow \mathsf{SL}(d, \mathbb{R})$  is a Benoist representation
if and only if $\Gamma$ has  cohomological dimension $d-1$ and  there is a $\rho$-equivariant continuous non-constant map 
$\xi:\partial\Gamma \rightarrow \mathbb{P}(\mathbb{R}^{d})$.
\end{theorem}

In Section \ref{examples} we discuss examples and questions related to our results.

\medskip\noindent
{\bf Remark:} One can obtain versions of all the results above in the
case when $\Gamma$ has torsion.
We recall that a representation is $P_k$-Anosov if and only
if its restriction to a finite index subgroup is $P_k$-Anosov and that finitely generated linear groups have finite index torsion-free
subgroups. It follows that if $\Gamma$ is not assumed to be torsion-free in the statements of Theorems \ref{sl3rcase} and \ref{borel4},
one can still conclude that $\rho(\Gamma)$ has a finite index subgroup  which is a free group or a surface group, while
in Theorem \ref{sl4rcase} one can conclude that $\rho(\Gamma)$ has a finite index subgroup isomorphic to a convex cocompact
subgroup of $\mathsf{PO}(3,1)$. In Theorem \ref{cohomdimd-k} one can conclude that $\rho(\Gamma)$ has the same
bounds on its virtual cohomological dimension, which one obtains on the cohomological dimension of $\Gamma$ in the 
torsion-free setting. If $\Gamma$ has a finite index torsion-free subgroup, one gets the same bounds on the virtual
cohomological dimension of $\Gamma$. In Theorems \ref{cohomdimd}  and \ref{benoistlimitmap} one must replace
the assumption that $\Gamma$ has cohomological dimension $d-1$ with the assumption that $\rho(\Gamma)$ has
virtual cohomological dimension $d-1$.

\medskip\noindent
{\bf Acknowledgements:} The authors would like to thank Fanny Kassel, Andres Sambarino, Michah Sageev,
Ralf Spatzier, Tengren Zhang and Andrew Zimmer for helpful conversations during the course of this work.

\section{Background}

\subsection{Anosov representations}
We briefly recall the definition of an Anosov representation and its crucial properties.
We first set some notation.
If $\Gamma$ is a hyperbolic group, we will fix a finite generating set for $\Gamma$
and let $|\gamma|$ denote the minimal word length of $\gamma$ and let $||\gamma||$ denote
the minimal word length of an element conjugate to $\gamma$ (i.e. the minimal translation
length of the action of $\gamma$ on $\Gamma$). Let $\partial \Gamma$ denote the
Gromov boundary of $\Gamma$. If $A\in\mathsf{GL}(d,\mathbb R)$, we let 
$$\lambda_1(A)\ge\lambda_2(A)\ge\cdots\ge\lambda_d(A)$$
denote the ordered moduli of the eigenvalues of $A$ (with multiplicity).

We will use a recent theorem of Kassel-Potrie \cite{kassel-potrie} (see also \cite[Thm 4.3]{kassel}) to give a simple definition
of $P_k$-Anosov representations.  If $\Gamma$ is a hyperbolic group and $1\le k\le\frac{d}{2}$
a representation $\rho:\Gamma\to\mathsf{GL}(d,\mathbb R)$ is  said to be {\em $P_k$-Anosov} if 
there exist constants $\mu,C>0$ so that 
$$\frac{\lambda_k(\rho(\gamma))}{\lambda_{k+1}(\rho(\gamma))}\ge Ce^{\mu||\gamma||}$$
for all $\gamma\in\Gamma$. A representation $\rho:\Gamma\to\mathsf{SL}(d,\mathbb R)$ is said to be {\em Anosov} 
if it is $P_k$-Anosov for
some $1\le k\le\frac{d}{2}$ and is said to be {\em Borel Anosov} if it is $P_k$-Anosov for every
$1\le k\le\frac{d}{2}$. It follows immediately from this definition that Anosov representations have
discrete image and finite kernel. 
(This definition is based on earlier definitions in terms of singular values due to Kapovich-Leeb-Porti \cite{KLP2},
Gu\'eritaud-Guichard-Kassel-Wienhard \cite{GGKW} and Bochi-Potrie-Sambarino \cite{BPS}.)

If a representation is $P_1$-Anosov, we call it {\em projective Anosov}. Projective Anosov representations are
in some sense the most general class of Anosov representations. If $\rho:\Gamma\to \mathsf{G}$ is an Anosov
representation into any semisimple Lie group, then there exists an irreducible representation 
$\tau:\mathsf{G}\to\mathsf{SL}(d,\mathbb R)$ so that $\tau\circ\rho$ is projective Anosov
(see \cite[Prop. 4.3]{guichard-wienhard}).

If a representation $\rho:\Gamma\to\mathsf{GL}(d,\mathbb R)$ is $P_k$-Anosov representation,
then it admits \hbox{$\rho$-equivariant} continuous maps 
$$\xi_\rho^k:\partial\Gamma\to {\rm Gr}_k(\mathbb R^d)\ \ \ {\rm and}\ \ \ \xi_\rho^{d-k}:\partial\Gamma\to {\rm Gr}_{d-k}(\mathbb R^d)$$
into the Grassmanian of $k$-planes and $(d-k)$-planes in $\mathbb R^d$. (If $d=\frac{k}{2}$, then $\xi_\rho^k=\xi_\rho^{d-k}$.)
If $\gamma^+\in\partial\Gamma$ is
the attracting fixed point of $\gamma$, then $\xi_\rho^k(\gamma^+)$ and $\xi_\rho^{d-k}(\gamma^+)$ are the attracting
$k$-planes and $(d-k)$-planes for $\rho(\gamma)$. Moreover, if $x,y\in\partial\Gamma$ are distinct, then
$$\xi_\rho^k(x)\subset\xi_\rho^{d-k}(x)\ \ \ {\rm and}\ \ \ \xi_\rho^k(x)\oplus\xi_\rho^{d-k}(y)=\mathbb R^d.$$
If $1\le j<k\le \frac{d}{2}$ and $\rho$ is both $P_j$-Anosov and $P_k$-Anosov, then
$$\xi_\rho^j(x)\subset\xi_\rho^k(x)\subset\xi_\rho^{d-k}(x)\subset\xi_\rho^{d-j}(x)$$
for all $x\in\partial\Gamma$.
(See Guichard-Wienhard \cite[Section 2]{guichard-wienhard} for a careful discussion of limit maps of Anosov representations.)

If $\rho:\Gamma\to\mathsf{SL}(d,\mathbb R)$ is a $P_k$-Anosov representation we may define a sphere bundle
$$E_{\rho,k}=\bigcup_{x\in\partial\Gamma} S(\xi^k(x))\subset S(\mathbb R^d)$$
over $\partial \Gamma$ where if $V$ is subspace of $\mathbb R^d$, then $S(V)$ is the unit sphere of $V$.
The bundle map $p_{\rho,k}:E_{\rho,k}\to \partial\Gamma$ is defined so that
$p_{\rho,k}(S(\xi^k(x))=x$. Notice that $p_{\rho,k}$ is well-defined, since, by transversality, $S(\xi^k(x))$ is disjoint from $S(\xi^k(y))$
if $x\ne y$.

\begin{lemma}
\label{fibrebundle}
If $\rho:\Gamma\to\mathsf{SL}(d,\mathbb R)$ is a $P_k$-Anosov representation, then
$p_{\rho,k}:E_{\rho,k}\to \partial\Gamma$ is a fibre bundle with fibres homeomorphic to $S^{k-1}$
In particular,  if $\partial\Gamma$ has topological dimension $m$, then $E_{\rho,k}$ has topological dimension $m+k-1$.
\end{lemma}

\begin{proof}
If $y\ne x$, let $\pi_{x,y}:\xi^k(y)\to\xi^{k}(x)$ denote orthogonal projection and notice that there exists an open neighborhood $U_x$ of $x$ in $\partial\Gamma$ such that
if $y\in U_x$ then $\pi_{x,y}$ is an isomorphism. Then there is a homeomorphism $\phi_x:p_{\rho,k}^{-1}(U_x)\to U_x\times S(\xi^k(x))$
where
$$\phi_x(z)=\left(p_{\rho,k}(z),\frac{\pi_{x,p(z)}(z)}{||\pi_{x,p(z)}(z)||}\right).$$
Therefore, $p_{\rho,k}$ is a fibre bundle with fibres homeomorphic to $S^{k-1}$.
Since $E_{\rho,k}$ is locally a topological product of $\mathbb R^{k-1}$ and a 
compact Hausdorff space $\partial\Gamma$ of topological dimension $m$,
it has topological dimension $m+k-1$ (see \cite{hurewicz-dimension}).
\end{proof}

\subsection{Semisimple representations}

We recall that a representation $\rho:\Gamma \rightarrow \mathsf{SL}(d, \mathbb{R})$ is said to be {\em semisimple}
if the Zariski closure of $\rho(\Gamma)$ in $\mathsf{SL}(d, \mathbb{R})$ is a reductive real algebraic Lie group. 
Moreover, if $\rho$ is semisimple, there exists a decomposition $\mathbb{R}^d=V_{1} \oplus...\oplus V_{k}$ 
into irreducible $\rho(\Gamma)$-modules.

If $\rho$ is a semisimple projective Anosov representation, then the restriction of
$\rho$ to the subspace spanned by the image of its limit map is irreducible.

\begin{proposition}
\label{limitmapspans}
If $\rho:\Gamma \rightarrow \SL$ is a semisimple projective Anosov representation,
then the action of $\rho(\Gamma)$ on the subspace
$W\subset\mathbb R^d$ spanned by $\xi_\rho^1(\partial\Gamma)$ is irreducible. 
Moreover,  if  $\rho_{W}:\Gamma\to \mathsf{SL}^\pm(W)$ is given by $\rho_W(\gamma)=|{\rm det}(\rho(\gamma)|_W)|^{-1/{\rm dim}(W)}\rho(\gamma)|_W$,
then $\rho_W$ is projective Anosov.
 \end{proposition}

\begin{proof} 
Let $\mathbb{R}^d=V_{1} \oplus...\oplus V_{k}$ be the decomposition into irreducible $\rho(\Gamma)$-modules.
Let $\gamma$ be an infinite order element of $\Gamma$. Then, since $\rho$ is projective Anosov, 
$\rho(\gamma)$ is proximal and $\xi_\rho^1(\gamma^+)$ is the
attracting eigenline of $\rho(\gamma)$. Since $\xi_\rho^1(\gamma^+)$ is an attracting eigenline, it must be contained in one of the $V_i$, 
so we may assume $\xi_\rho^1(\gamma^+)\subset V_1$.
Since $V_{1}$ is $\rho(\Gamma)$-invariant we have that $\xi_\rho^1(\alpha(\gamma^+))=\rho(\alpha)(\xi_\rho^1(\gamma^{+})) \subset V_{1}$
for all $\alpha\in\Gamma$.
Moreover, since the orbit of $\gamma^+$ is dense in $\partial\Gamma$, we conclude that 
that $\left \langle \xi(\partial_{\infty}\Gamma) \right \rangle$ is a $\rho(\Gamma)$-invariant vector subspace of $V_{1}$. 
The restriction of $\rho$ to $V_1$ is irreducible, so it follows that $\left \langle \xi_\rho^1(\partial\Gamma) \right \rangle=V_{1}$.  

Notice that $\rho_W$ is projective Anosov,  since
$\rho$ is projective Anosov and 
\hbox{$\lambda_1(\rho(\gamma)|_W)=\lambda_1(\rho(\gamma))$} and
$\lambda_2(\rho(\gamma)|_W)\le\lambda_2(\rho(\gamma))$ for all $\gamma\in\Gamma$.
\end{proof}

Benoist \cite{benoist-limitcone}  used work of Abels-Margulis-Soifer \cite{AMS} to establish
the following useful relationship between eigenvalues and singular values for semisimple
representation (see \cite[Thm 4.12]{GGKW}  for a proof).

\begin{theorem}
\label{finitesubset}
If $\Gamma$ is a finitely generated group and $\rho:\Gamma\to \mathsf{SL}(d,\mathbb R)$ is semisimple,
then there exists a finite subset $A$ of $\Gamma$ and $M>0$ so
that if $\gamma\in\Gamma$, then there exists $\alpha\in A$ so that
$$|\log\lambda_i(\rho(\gamma \alpha))-\log\sigma_i(\rho(\gamma))|\le M$$
for all $i$.
\end{theorem}

Gueritaud, Guichard, Kassel and Wienhard \cite[Section 2.5.4]{GGKW} observe that given any representation $\rho:\Gamma \to \SL$, one may 
define a {\em semisimplification} $\rho^{ss}:\Gamma\to\SL$. They further show that the Jordan projections agree,  so
$\rho$ and $\rho^{ss}$ share the same Anosov qualities.

\begin{lemma}
\label{semisimpleproperties}
{\rm (\cite[Prop 2.39, Lemma 2.40]{GGKW})}
If $\rho:\Gamma\to \mathsf{SL}(d,\mathbb R)$ is a representation of a hyperbolic group and $\rho^{ss}$ is a semisimplification of $\rho$,
then $\lambda_i(\rho^{ss}(\gamma))=\lambda_i(\rho(\gamma))$ for all $i$ and all $\gamma\in\Gamma$.
In particular, $\rho$ is $P_k$-Anosov if and only if $\rho^{ss}$ is $P_k$-Anosov.
\end{lemma}

Proposition \ref{limitmapspans}  and Lemma \ref{semisimpleproperties} have the following useful corollary.

\begin{corollary}
\label{reducible}
If $\rho:\Gamma \rightarrow \mathsf{SL}(d,\mathbb R)$ is a reducible projective Anosov representation and $\rho^{ss}$ is its
semisimplification, then the action of $\rho^{ss}(\Gamma)$ on 
the proper subspace $W$ of $\mathbb R^d$ spanned by $\xi^1_{\rho^{ss}}(\partial\Gamma)$ is irreducible. 
Moreover, $\rho^{ss}_W$ is projective Anosov.
 \end{corollary}

\subsection{Convex cocompactness}

Danciger-Gu\'eritaud-Kassel \cite{DGK} and Zimmer \cite{zimmer}  have shown that many projective Anosov representations
can be understood as convex cocompact actions on properly convex domains in projective space.
We recall that a domain $\Omega\subset\mathbb P(\mathbb R^d)$ is said to be {\em properly convex} if
it is a bounded subset of an affine chart $A=\mathbb P(\mathbb R^d-V)$ where $V$ is a $(d-1)$-plane in
$\mathbb R^d$ and $\Omega$ is convex in $A$. The domain $\Omega$ is {\em strictly convex} if it is a bounded,
strictly convex subset of some affine chart.
We say that $\rho(\Gamma)$ is a convex cocompact
subgroup of ${\rm Aut}(\Omega)$ if $\rho(\Gamma)$ preserves $\Omega$ and there is a closed convex 
$\rho(\Gamma)$-invariant subset $C$ of
$\Omega$ so that $C/\rho(\Gamma)$ is compact. 
(See \cite{DGK} or \cite{zimmer} for details.)

\begin{theorem}
\label{Zimmer}
{\rm (Zimmer \cite[Thm 1.25]{zimmer})}
Suppose that $\Gamma$ is a torsion-free hyperbolic group with connected boundary $\partial\Gamma$ which is
not a surface group and
 $\rho:\Gamma\to \mathsf{SL}(d,\mathbb R)$ is an irreducible projective Anosov
representation. Then there exists a properly convex domain $\Omega\subset\mathbb P(\mathbb R^d)$
so that $\rho(\Gamma)$ is a convex cocompact subgroup of ${\rm Aut}(\Omega)$.
\end{theorem}

We recall that a hyperbolic group has connected boundary if and only if it is one-ended.

Danciger, Gueritaud and Kassel \cite{DGK} describe the maximal domain that a convex cocompact group acts on.

\begin{proposition}
\label{Omegastructure}
{\rm (Danciger-Gueritaud-Kassel \cite[Cor. 8.10]{DGK})}
Suppose $\rho:\Gamma\to \mathsf{SL}(d,\mathbb R)$ is an irreducible projective Anosov
representation, $\partial\Gamma$ is connected and $\rho(\Gamma)$ preserves a properly
convex open subset of $\mathbb P(\mathbb R^d)$. Then $\rho(\Gamma)$ acts convex cocompactly on
$$\Omega_{max}=\mathbb P(\mathbb R^d)-\bigcup_{x\in\partial\Gamma}\mathbb P(\xi_\rho^{d-1}(x))$$
and $\xi_\rho^1(\partial\Gamma)\subset \partial\Omega_{max}$.
\end{proposition}

We will also need the following result which is implicit in Zimmer's work \cite{zimmer}.

\begin{proposition}
\label{Zimmerconnected}
Suppose that $\Gamma$ is a hyperbolic group with connected boundary $\partial\Gamma$, 
\hbox{$\rho:\Gamma\to\mathsf{SL}(d,\mathbb R)$}
is a representation, and \hbox{$\xi:\partial\Gamma\to \mathbb P(\mathbb R^d)$} is a $\rho$-equivariant continuous map.
If $\xi(\partial \Gamma)$ spans $\mathbb R^d$ and lies inside
an affine chart, then there exists a properly convex domain $\Omega\subset\mathbb P(\mathbb R^d)$
so that $\rho(\Gamma)$ preserves $\Omega$.
\end{proposition}

\begin{proof} 
By assumption $S=\xi(\partial\Gamma)$ is a connected subset of an affine chart
$A \subset \mathbb{P}(\mathbb{R}^{d})$. We may assume that
$$A=\left \{ \left [ x_1:...:x_d \right ]: x_1 \neq 0  \right \}$$
so every point in $A$ has a unique representative of the form $[1:u]$ where $u\in\{0\}\times \mathbb R^{d-1}$.
Then 
$$CH(S)=\left\{[1:u] \ \Big| \ u=\sum_{i=1}^{d}t_i u_i, \  \sum_{i=1}^{d} t_i=1, \ t_i\ge 0\ {\rm and}\ [1:u_{i}] \in S\ {\rm for}\ {\rm all}\ i \right\}$$
is the convex hull of $S$ in the affine chart $A$. Since  $S$ spans $\mathbb R^d$,
$CH(S)$ has non-empty interior.

It only remains to show that $CH(S)$ is $\Gamma$-invariant. Suppose that $\gamma\in\rho(\Gamma)$.
If $[1:u]\in S$, then $\gamma([1:u])=[\gamma(e_1)+\gamma(u)]\in S\subset A$, so
$$\left \langle \gamma(e_1)+\gamma(u),e_1 \right \rangle \neq 0.$$
Since $S$ is connected, $\left \langle \gamma(e_1)+\gamma(u),e_1 \right \rangle$ always has the same sign if $[1:u]\in S$.
If 
$$x=\sum_{i=1}^d t_i[1:u_i]\in CH(S)$$
 let
$$B_i(x)=\left\langle\gamma(e_1)+\gamma(u_i),e_1 \right \rangle\ \ {\rm and}\ \ B(x)=\sum_{i=1}^d B_i(x),$$ 
so
\begin{eqnarray*}
\gamma(x) &= &\gamma\left(\sum_{i=1}^{d}t_i[1,u_i] \right)=\Big[\gamma(e_1)+\sum_{i=1}^{d}t_i\gamma(u_i) \Big]=
\Big[B(x)e_1+\sum_{i=1}^{d}t_i\Big(\gamma(u_i)+\gamma(e_1)-B_i(x) e_1 \Big) \Big]\\
&=&\left[ e_1+\sum_{i=1}^{d} \frac{t_iB_i(x)}{B(x)}\left( \frac{1}{B_i(x)}\left(\gamma(u_i)+\gamma(e_1)-B_i(x) e_1\right) \right) \right]\\
&=&\sum_{i=1}^d s_i \left[1:\frac{\gamma(e_i)+\gamma(u_i)-B_i(x) e_1}{B_i(x)}\right]=\left[\sum_{i=1}^d s_i \gamma([1:u_i])\right]\\
\end{eqnarray*}
where
$$s_{i}= \frac{t_iB_i(x)}{B(x)}.$$
Notice that $s_i\ge 0$  for all $i$, since $t_i\ge 0$ and $B_i(x)$ and $B(x)$ have the same sign.
Therefore, $\gamma(x)\in CH(S)$, so $CH(S)$ is $\Gamma$-invariant as required.
\end{proof}

We combine Proposition \ref{Zimmerconnected} and Proposition \ref{Omegastructure} to show that an
irreducible projective Anosov surface group whose limit set lies in an affine chart preserves
a properly convex domain of the form given by Proposition \ref{Omegastructure}.

\begin{corollary}
\label{Omegasurface}
Suppose that $\Gamma$ is a surface group and $\rho:\Gamma\to \mathsf{SL}(d,\mathbb R)$ is a projective Anosov
representation. If $\xi_\rho^1(\partial \Gamma)$ spans $\mathbb R^d$ and lies inside
an affine chart for $\mathbb R^d$, then $\rho(\Gamma)$ acts convex cocompactly on the properly convex domain
$$\Omega=\mathbb P(\mathbb R^d)-\bigcup_{x\in\partial\Gamma}\mathbb P(\xi_\rho^{d-1}(x))$$
and $\xi^1_\rho(\partial\Gamma)\subset \partial\Omega$.
\end{corollary}

\section{Anosov representations into $\mathsf{SL}(3,\mathbb R)$}

We first characterize Anosov representations into $\mathsf{SL}(3,\mathbb R)$.

\medskip\noindent
{\bf Theorem \ref{sl3rcase}:} {\em
If $\Gamma$ is a torsion-free hyperbolic group and $\rho:\Gamma \to \mathsf{SL}(3,\mathbb R)$ is an Anosov representation,
then  $\Gamma$ is either a free group or a surface group.
}

\begin{proof}
Notice that an Anosov representation into $\mathsf{SL}(3,\mathbb R)$ is, by  definition, projective Anosov.
If $\Gamma$ is not free or a surface group, let $\Gamma=\Gamma_1 \ast \cdots\ast \Gamma_{r}\ast F_{s}$ 
be the free product decomposition of $\Gamma$ where each $\Gamma_i$ is one-ended, $r\ge 1$ and $s \geqslant 0$. 
Recall that $\Gamma_1$ is quasiconvex in $\Gamma$ (see  \cite[Prop. 1.2]{bowditch-cutpoints}), 
so $\rho|_{\Gamma_1}$ is projective Anosov \cite[Lemma 2.3]{CLS}.

If $\rho|_{\Gamma_1}$ is reducible,  then Corollary \ref{reducible} implies that there exists a projective
Anosov representation of $\Gamma_1$ into $\mathsf{SL}^\pm(W)$ where $W$ is a proper subspace of $\mathbb R^3$.
However, every torsion-free discrete subgroup of $\mathsf{SL}^\pm(W)$ is either a free group or a surface group
if $W$ is one or two dimensional.

If $\rho|_{\Gamma_1}$ is irreducible and $\Gamma_1$ is not a surface group, then 
Theorem \ref{Zimmer}  implies that $\rho(\Gamma)$ acts convex cocompactly on a properly convex domain 
$\Omega\subset \mathbb P(\mathbb R^3)$, but then
$\Gamma$ is isomorphic to the fundamental group of the surface $\Omega/\rho(\Gamma)$ which is  a contradiction.
We conclude that $\Gamma_1$ is a surface group.

Suppose that $\Gamma_1$ has infinite index in $\Gamma$ (i.e. suppose that there is more than one factor).
Then $\partial\Gamma_1$ is a proper subset of $\partial \Gamma$. Let $z\in\partial\Gamma-\partial\Gamma_1$.
Notice that $\xi_\rho(\partial \Gamma_1)$ is a compact
subset of the affine chart \hbox{$\mathbb{P}(\mathbb{R}^3-\xi^2_\rho(z))$}. 
Since $\xi_\rho^1(\partial\Gamma_1)$ is an embedded circle in $A$,
$\xi_\rho^1(\partial\Gamma_1)$ must span $\mathbb R^3$ (otherwise, $\xi_\rho^1(\partial\Gamma)$
would be contained in the intersection of a projective line with $A$). Corollary \ref{Omegasurface}
then implies that $\rho(\Gamma_1)$ acts cocompactly on 
$$\Omega=\mathbb{P}(\mathbb{R}^3)-\bigcup_{x\in \partial\Gamma_1}\mathbb{P}(\xi^2_\rho(x))$$
and that $\partial\Omega=\xi^1_\rho(\partial\Gamma_1)$.
In particular, either $\xi^1_\rho(z)$ is  contained in $\xi^2_\rho(x)$ for some $x\in\partial\Gamma_1$, which violates transversality,
or $\xi_\rho^1(z)$ is contained in $\Omega$ which would imply that  $\xi^2_\rho(z)$ must
intersects $\partial\Omega=\xi_\rho^1(\partial\Gamma_1)$,
which agains violates transversality. This final contradiction completes the proof.
\end{proof}

\section{Projective Anosov representations into $\mathsf{SL}(4,\mathbb R)$}

The inclusion of any convex cocompact subgroup of $\mathsf{SO}(3,1)$ into $\mathsf{SL}(4,\mathbb R)$ 
is a projective Anosov representation (see \cite[Sec. 6.1]{guichard-wienhard}). We use work of
Danciger-Gu\'erituad-Kassel \cite{DGK} and Zimmer \cite{zimmer} and the Geometrization
Theorem to show that these are the only groups admitting projective Anosov representations into $\mathsf{SL}(4,\mathbb R)$.

\medskip\noindent{\bf Theorem \ref{sl4rcase}:}
{\em If $\Gamma$ is a torsion-free  hyperbolic group and  $\rho:\Gamma\to  \mathsf{SL}(4,\mathbb R)$ is a projective Anosov representation, 
then $\Gamma$ is isomorphic to a convex cocompact subgroup of $\mathsf{PO}(3,1)$.
}

\medskip\noindent
{\em Proof of Theorem \ref{sl4rcase}:}
%
We decompose $\Gamma=\Gamma_1*\cdots*\Gamma_p*F_r$ where each $\Gamma_i$ is one-ended.
Notice that $F_r$ is the fundamental group of a handlebody of genus $r$, which
is a compact irreducible 3-manifold with non-empty boundary.
If $\Gamma_i$ is a surface group, then it is the fundamental group of an  interval bundle $M_i$, so $M_i$ is irreducible and has 
non-empty boundary.

Now suppose that $\Gamma_i$ is not a surface group.
Since $\Gamma_i$ is a quasiconvex subgroup of $\Gamma$ (see  \cite[Prop. 1.2]{bowditch-cutpoints}), 
$\rho_i=\rho|_{\Gamma_i}$ is also projective Anosov (see \cite[Lemma 2.3]{CLS}).
If $\rho_i$ is reducible,
then Corollary \ref{reducible} gives a proper subspace $W$ of $\mathbb R^4$ and an irreducible projective Anosov
representation 
$$\hat\rho_i=(\rho_i)_W^{ss}:\Gamma_i\to\mathsf{SL}^\pm(W).$$
If $W$ is one-dimensional, then $\Gamma_i$ is not one-ended. If $W$ is two-dimensional, then
$\hat\rho_i$ is a Fuchsian representation, so $\Gamma_i$ is either a free group or a surface group, both of  which
have been disallowed. If $W$ is 3-dimensional, then Theorem \ref{sl3rcase} again implies that 
$\Gamma_i$ is either a free group or a surface group, which is impossible. Therefore, $\rho_i$ is
irreducible.

Since $\rho_i$ is irreducible and $\Gamma_i$ is one-ended and not a surface group, 
Theorem \ref{Omegastructure} implies that $\rho(\Gamma_i)$
acts convex cocompactly on a properly convex domain 
$$\Omega_i=\mathbb P\left(\mathbb R^4-\bigcup_{x\in\partial\Gamma_i}\xi_\rho^{3}(x)\right)\subset\mathbb P(\mathbb R^4).$$
In particular, $\Gamma_i$ is the fundamental group of an irreducible 3-manifold $N_i=\Omega_i/\Gamma_i$.  

If $N_i$ is a closed $3$-manifold, then $\rho_i$ is a Benoist representation. If $\Gamma\ne \Gamma_i$, then there
exists $z\in\partial\Gamma-\partial\Gamma_i$. By transversality, $\xi_\rho^1(z)$ cannot lie in $\mathbb P(\mathbb R^4)-\Omega_i$.
On the other hand, if $z\in\Omega_i$, then $\xi_\rho^3(z)$ intersects $\partial\Omega_i=\xi_\rho^1(\partial\Gamma_i)$,
which also violates transversality. Therefore, if $N_i$ is closed, $\Gamma=\Gamma_i$,
$\rho$ is a Benoist representation and the Geometrization
Theorem \cite{morgan-tian} implies that  $\Gamma$ is isomorphic to a convex cocompact subgroup of $\mathsf{PO}(3,1))$.

If $N_i$ is not a closed $3$-manifold, then
the Scott core theorem \cite{scott} implies that $N_i$
contains a compact, irreducible submanifold $M_i$ with non-empty boundary and fundamental group $\Gamma_i$.

Therefore, if $\rho$ is not a Benoist representation, $\Gamma$ is the fundamental group of the boundary connected sum 
$M$ of the $M_i$ and a handlebody of genus $r$. Since $M$ is irreducible, $\pi_1(M)$ is word hyperbolic, torsion-free and infinite, 
and $M$ has non-empty boundary,
it follows from Thurston's original Geometrization Theorem (see Morgan \cite{morgan})
that the interior of $M$ admits a convex cocompact hyperbolic
structure, so $\Gamma$ is isomorphic to a convex cocompact subgroup of $\mathsf{PO}(3,1)$.
\eproof

\section{Cohomological resrictions on $P_k$-Anosov representations}

In this section, we place cohomological restrictions on $P_k$-Anosov subgroups of $\mathsf{SL}(d,\mathbb R)$.  
Benoist \cite[Prop 1.3]{benoist-divisible3} previously proved that a discrete hyperbolic subgroup of $\mathsf{SL}(d,\mathbb R)$
consisting entirely of positively semi-proximal elements has virtual cohomological dimension at most $d-1$, with equality if and only
if the inclusion map is a Benoist representation. Guichard and Wienhard \cite[Prop. 8.3]{guichard-wienhard} obtained
bounds on cohomological dimension for various classes of Anosov representations into specified Lie subgroups of
$\mathsf{SL}(d,\mathbb R)$.

\medskip\noindent
{\bf Theorem \ref{cohomdimd-k}:} {\em
Suppose $\Gamma$ is a torsion-free hyperbolic group and $\rho:\Gamma\to  \mathsf{SL}(d,\mathbb R)$ is \hbox{$P_k$-Anosov.}
\begin{enumerate}
\item
If $(d,k)$ is not  $(2,1)$, $(4,2)$, $(8,4)$ or $(16,8)$, 
then $\Gamma$ has  cohomological dimension at most $d-k$.
\item
If $(d,k)$ is $(2,1)$, $(4,2)$, $(8,4)$ or $(16,8)$, then $\Gamma$ has virtual cohomological dimension at most $d-k+1$. 
Moreover, if $\Gamma$ has  cohomological dimension $d-k+1$ and $(d,k)$ is $(4,2)$, $(8,4)$ or $(16,8)$,
then $\partial \Gamma$ is homeomorphic to $S^{d-k}$ and $\rho$ is not  projective Anosov.
\end{enumerate}
}

\medskip\noindent
{\em Proof of Theorem \ref{cohomdimd-k}:}
If $d=2$, then $k=1$ and $\rho:\Gamma\to \mathsf{SL}(2,\mathbb R)$ is Fuchsian, so either $\Gamma$
has cohomological dimension $d-k=1$ (if $\Gamma$ is free) or $2$ (if $\Gamma$ is a surface group),
which corresponds to the first exceptional case in item (2).
Theorem \ref{sl3rcase} handles the case where $d=3$.

Now suppose that $\frac{d}{2}\ge k\ge 1$ and $d>3$. Let $m$ be the topological dimension of $\partial\Gamma$.
Fix $x_0\in \partial\Gamma$ and a $(d-k+1)$-plane $V$ in $ \mathbb R^d$ which contains $\xi^{d-k}(x_0)$. We define a map 
$$F:\partial \Gamma -\{x_0\} \to \mathbb P(V-\xi^{k}(x_0))$$
by letting $F(y)$ be the line which is the intersection of $\xi^k(y)$ with $V$.
(Transversality implies that the intersection of $\xi^k(y)$ and $\xi^{d-k}(x_0)$ is trivial if $y\ne x_0$,
so the intersection of $\xi^k(y)$ with $V$ must be a line.)

One sees that $F$ is injective, since if $x\ne y\in\partial\Gamma$, then $\xi^k(x)$ and $\xi^k(y)$
have trivial intersection (by transversality). Moreover, $F$ is proper, since if $\{y_n\}$ is a sequence in
$\partial\Gamma-\{x_0\}$ converging to $x_0$, then, by continuity of limit maps, $\{\xi_\rho^k(y_n)\}$
is converging to $\xi_\rho^k(x_0)$, so $\{F(y_n)\}$ leaves every compact subset of  $\mathbb P(V-\xi^{k}(x_0))$.
Therefore, $F$ is an embedding.
Since $\partial\Gamma-\{x_0\}$ embeds in a $(d-k)$-manifold, $\partial \Gamma$
has topological dimension at most $d-k$ (see \cite[Thm III.1]{hurewicz}).

Now suppose that  $\partial\Gamma$ has topological dimension exactly $d-k$. Then,
$F(\partial\Gamma)$ contains an open subset of $\mathbb P(V)$ (see \cite[Thm. IV.3/Cor. 1]{hurewicz}).
So, since $\partial\Gamma$ has a manifold point, $\partial\Gamma$ is homeomorphic to $S^{d-k}$,
by
Kapovich-Benakli \cite[Th. 4.4]{benakli-kapovich}.
Let $p:E\to \partial\Gamma$ be the fibre bundle provided by Lemma \ref{fibrebundle}
where 
$$E=\bigcup_{x\in\partial\Gamma} S(\xi^k(x))\subset S(\mathbb R^d).$$
Then $E$ has topological dimension $(d-k)+k-1=d-1$. Since $\partial\Gamma$ is homeomorphic to $S^{d-k}$,
$E$ is a closed submanifold of $S(\mathbb R^d)\cong S^{d-1}$ of dimension $d-1$, which implies that $E=S(\mathbb R^d)$.
However, by the classification of sphere fibrations (\cite{adams}), 
this is only possible if $(d-1,k-1)$ is $(3,1)$, $(7,3)$ or $(15,7)$.
Moreover, in these cases, $\rho$ cannot be projective Anosov, since if $\rho$ is projective Anosov,
$\xi_\rho^1:\partial\Gamma\to \mathbb P(\mathbb R^{2k})$ lifts to 
a section $s:\partial\Gamma\to E$ of $p$,
which is impossible  (since $p\circ s=id$, $p_*\circ s_*$ is the identity
map on $H_{d-k}(S^{d-k})\cong \mathbb Z$, while, $p_*$ is the zero map
on $H_{d-k}(E)$.)

If $\partial\Gamma$ has topological dimension  at most $d-k-1$, then, by Bestvina-Mess \cite[Cor. 1.4]{bestvina-mess}, $\Gamma$
has cohomological dimension at most $d-k$. If $\partial\Gamma$ has topological dimension $d-k$, then
$\Gamma$ has cohomological dimension $d-k+1$, again by \cite[Cor. 1.4]{bestvina-mess}, and, by the previous paragraph,
$(d,k)$ is $(2,1)$, $(4,2)$, $(8,4)$ or $(16,8)$, $\partial\Gamma\cong S^{d-k}$
and $\rho$ is not projective Anosov if $(d,k)$ is $(4,2)$, $(8,4)$ or $(16,8)$.
\eproof

\section{Benoist representations}

The prototypical example of a Benoist representation is  the inclusion of a cocompact discrete subgroup of
$\mathsf{PO}(n,1)$ into $\mathsf{SL}(n+1,\mathbb R)$. The image acts convex cocompactly on a round disk in $\mathbb P(\mathbb R^{n+1})$
which is the Beltrami-Klein model for $\mathbb H^n$. Johnson and Millson \cite{johnson-millson} showed
that although the inclusion map is rigid in $\mathsf{PO}(n,1)$ if $n\ge 3$, it often admit non-trivial deformations
in $\mathsf{SL}(n+1,\mathbb R)$. Benoist \cite{benoist-divisible1,benoist-divisible2,benoist-divisible3} showed that these deformations are always Benoist representations
and developed an extensive theory of groups of projective automorphisms preserving properly convex subsets of 
$\mathbb P(\mathbb R^d)$.

If $\rho$ is a Benoist representation, then $\rho(\Gamma)$ has virtual cohomological dimension $d-1$, so,
Theorem \ref{cohomdimd-k}  immediately implies that Benoist representations are only projective Anosov.

\medskip\noindent
{\bf Corollary \ref{benoistnotpk}:} {\em
If  $\rho:\Gamma\to \mathsf{SL}(d,\mathbb R)$ is a Benoist representation, and $k$ is an integer such that
$\frac{d}{2}\ge k\ge 2$, then $\rho$ is not $P_k$-Anosov. }

\medskip

We characterize Benoist representations in terms of the cohomological dimension of their domain groups.

\medskip\noindent
{\bf  Theorem \ref{cohomdimd}:} {\em
If $d\ge 4$, an Anosov representation $\rho:\Gamma\to\mathsf{SL}(d,\mathbb R)$ of a torsion-free hyperbolic group $\Gamma$
is a Benoist representation if and only if 
$\Gamma$ has cohomological dimension $d-1$.}

\medskip\noindent
{\em Proof of Theorem \ref{cohomdimd}:}
It is immediate from the definition that if $\Gamma$ is torsion-free and \hbox{$\rho:\Gamma\to\mathsf{SL}(d,\mathbb R)$}
is a Benoist representation, then $\Gamma$ has cohomological dimension $d-1$.

Now suppose that $\Gamma$ is a torsion-free hyperbolic group of cohomological dimension $d-1$ and 
$\rho:\Gamma\to\mathsf{SL}(d,\mathbb R)$ is an Anosov representation.
Notice that, by Theorem \ref{cohomdimd-k}, $\rho$ cannot be $P_k$-Anosov for any $k\ge 2$,
so $\rho$ must be projective Anosov.

There exists a free decomposition 
$\Gamma=\Gamma_1 \ast\cdots\ast \Gamma_{s}\ast F_{r}$ where each $\Gamma_i$ is one-ended.
Since  the cohomological dimension of $\Gamma$ is the maximum of the cohomological dimensions of its
one-ended factors, we may assume that $\Gamma_1$ has cohomological dimension  $d-1$.
Since $\Gamma_1$ is a quasiconvex subgroup of $\Gamma$ (see \cite[Prop. 1.2]{bowditch-cutpoints}), the restriction
$\rho_1=\rho|_{\Gamma_1}$ of $\rho$ to $\Gamma_1$ is
projective Anosov \cite[Lemma 2.3]{CLS}. Moreover, $\xi^1_{\rho_1}$ is the restriction of $\xi^1_\rho$ to $\partial\Gamma_1\subset\partial\Gamma$.

We first claim that  $\rho_1$ is irreducible. If not,
Corollary \ref{reducible} provides  a proper $\rho(\Gamma_1)$-invariant subspace $W$ of $\mathbb R^d$ and
an irreducible projective Anosov representation \hbox{$(\rho_1)^{ss}_W:\Gamma\to\mathsf{SL}^\pm(W)$}.  
However, Theorem \ref{cohomdimd-k} would then imply that $\Gamma$ has
cohomological dimension at most \hbox{${\rm dim}(W)-1\le d-2$}, which would be a contradiction.

Since $\rho_1$ is irreducible and $\Gamma$ is one-ended and not a surface group, 
Theorem \ref{Zimmer} implies that there exists a properly convex open domain 
$\Omega \subset \mathbb{RP}^d$ preserved by $\rho(\Gamma_1)$.  Proposition \ref{Omegastructure} implies
that we may assume that $\Omega$ has the form
$$\Omega=\mathbb{P}(\mathbb{R}^{d})-\bigcup_{x \in \partial\Gamma_1}\mathbb{P}(\xi^{d-1}_{\rho}(x)).$$ 
Since $\Omega/\rho(\Gamma_1)$ is an aspherical $(d-1)$-manifold and $\Gamma_1$ has cohomological
dimension $d-1$, $\Omega/\rho(\Gamma_1)$ must be a closed manifold. Therefore, by Benoist \cite[Thm 1.1]{benoist-divisible1},
$\Omega$ is strictly convex, so $\rho_1$ is a 
Benoist representation.

If there existed another one-ended or cyclic factor in the free decomposition of $\Gamma$, then there 
would exist an infinite order element $t$ in the other factor. In particular, $t^+$ does not lie in $\partial\Gamma_1$.
Since, by transversality, $\xi_\rho^1(t^+)$ cannot intersect $\xi^{d-1}_{\rho}(x)$ for any $x$ in $\partial\Gamma_1$,
$\xi_\rho^1(t^+)$ must lie in $\Omega$. However, if $\xi_\rho^1(t^{+})\in\Omega$, then the hyperplane $\xi^{d-1}_{\rho}(t^{+})$ must 
intersect the boundary $\partial \Omega$, which again violates transversality of the limit maps.
Therefore, $\Gamma=\Gamma_1$ and $\rho$ is a Benoist representation.
\eproof

\section{Borel Anosov representations}

The only known examples of Borel Anosov representations  into $\mathsf{SL}(d,\mathbb R)$ have domain groups which contain
finite index subgroups which are either free or surface groups. Andres Sambarino asked whether this is always the case.

\medskip\noindent
{\bf Sambarino's Question:} {\em If a torsion-free hyperbolic group  admits a Borel Anosov representations into $\mathsf{SL}(d,\mathbb R)$, must it 
be either a free group or a surface group?}

\medskip

We do know, by Theorem \ref{cohomdimd-k}, that Borel Anosov representations must have ``small''  cohomological dimension.

\begin{corollary}
Suppose that $\Gamma$ is a torsion-free hyperbolic group, $d\ge 3$ and $\rho:\Gamma\to \mathsf{SL}(d,\mathbb R)$ is Borel Anosov. 
\begin{enumerate}
\item
If $d$ is odd, then $\Gamma$ has cohomological dimension at most $\frac{d+1}{2}$.
\item
If $d$ is even, then $\Gamma$ has  cohomological dimension at most $\frac{d}{2}$. 
\end{enumerate}
\end{corollary} 

Theorem \ref{sl3rcase} answers the question in the positive when $d=3$. In this section, we  handle the
case when $d=4$. We first observe that every Borel Anosov representation of a surface group into $\mathsf{SL}(d,\mathbb R)$ is irreducible
.

\begin{proposition} 
\label{Borel4irreducible}
If $\Gamma$ is a surface group and $\rho:\Gamma \to \mathsf{SL}(4,\mathbb R)$ is Borel Anosov,
then $\rho$ is irreducible.
\end{proposition}

\begin{proof}{}
Suppose not. Then we may assume that $\rho$ is  a reducible, semisimple,  Borel Anosov,  representation,
since the semisimplification of $\rho$ remains reducible and Borel Anosov.
Let $W$ be the subspace spanned by $\xi_\rho^1(\partial \Gamma)$. Then the restriction of $\rho$ to $W$
is irreducible, by Proposition \ref{limitmapspans}. Let $V$ be the complementary subspace of $\mathbb R^4$ 
which is also preserved by $\rho(\Gamma)$.

The subspace $W$ cannot be 1-dimensional, since $\xi_\rho$ is injective.

If $W$ is three-dimensional, then $V$ is an eigenline of each $\rho(\gamma)$
so, for all $\gamma$, $V$ lies in either $\xi^2_\rho(\gamma^+)$ or in $\xi^{2}_\rho((\gamma^{-1})^+)$.
However, this is impossible since $\xi^2_\rho(\alpha^+)$ and $\xi^2_\rho(\beta^+)$ are transverse for all
$\alpha$ and $\beta$ in distinct maximal cyclic subgroups.

If $W$ is two-dimensional, then we may pass to a subgroup of index at most $4$, still called $\Gamma$,
so that $\rho(\gamma)|_V$ and $\rho(\gamma)|_W$ both have positive determinant for all $\gamma\in\Gamma$.
Let 
$$a(\gamma)=\sqrt{{\rm det}(\rho(\gamma)|_W)}$$
for all $\gamma\in\Gamma$ and define 
$\rho_1:\Gamma\to \mathsf{SL}(W)$ by $\rho_1(\gamma)=a(\gamma)^{-1}\rho(\gamma)|_W$
and $\rho_2:\Gamma\to \mathsf{SL}(V)$ by $\rho_2(\gamma)=a(\gamma)\rho(\gamma)|_V$. Since 
$$\frac{\lambda_1(\rho_1(\gamma))}{\lambda_2(\rho_1(\gamma))}=\frac{\lambda_1(\rho(\gamma))}{\lambda_4(\rho(\gamma))}\ \ \
{\rm and}\ \ \ \frac{\lambda_1(\rho_2(\gamma))}{\lambda_2(\rho_2(\gamma))}=\frac{\lambda_2(\rho(\gamma))}{\lambda_3(\rho(\gamma))}$$
for all $\gamma\in\Gamma$ and $\rho$ is Borel Anosov, we see that $\rho_1$ and $\rho_2$ are Fuchsian.

Since $\rho$ is projective Anosov, there exists $s>0$ so that 
$$\lambda_1(\rho(\gamma))>e^{s||\gamma||} \lambda_2(\rho(\gamma))$$
where $||\gamma||$ is the cyclically reduced word length of $\gamma$. Observe that
$$\lambda_1(\rho(\gamma))=a(\gamma)\lambda_1(\rho_1(\gamma))
\ \ \ {\rm and}\ \ \ \ \lambda_2(\rho(\gamma))=a(\gamma)^{-1}\lambda_1(\rho_2(\gamma))$$
for all $\gamma\in\Gamma$, so
$$\lambda_1(\rho_1(\gamma))>e^{s||\gamma||}a(\gamma)^{-2} \lambda_2(\rho(\gamma)).$$
Now, since $\lambda_1(\rho_i(\gamma^{-1}))=\lambda_1(\rho_i(\gamma))$ and $a(\gamma^{-1})=a(\gamma)^{-1}$,
we see that 
$$\lambda_1(\rho_1(\gamma))>e^{s||\gamma||} \lambda_2(\rho(\gamma))$$
for all $\gamma\in\Gamma$. 
However, this is impossible, since
$$h(\rho_i)=\lim\frac{1}{T}\log\ \#\{[\gamma]\in[\Gamma]\ |\ 2\log\lambda_1(\rho_i(\gamma))\le T\}=1$$
for both $i=1,2$ (see Patterson \cite{patterson}).
\end{proof}

We are now ready to answer Sambarino's question when $d=4$.

\medskip\noindent
{\bf Theorem \ref{borel4}:}
{\em If  $\Gamma$ is a torsion-free hyperbolic group and $\rho:\Gamma\to \mathsf{SL}(4,\mathbb R)$ is Borel Anosov,
then $\Gamma$ is either a surface group or a free group.}

\medskip\noindent 
{\em Proof of Theorem \ref{borel4}:}
By Theorem \ref{sl4rcase}, we know that $\Gamma$
is isomorphic to a convex cocompact subgroup $\Delta$ of $\mathsf{PO}(3,1)$. Moreover, by Theorem \ref{cohomdimd-k}, 
$\Gamma$ has cohomological dimension at most 2, so $\Delta$ is not cocompact. Therefore, if $\Gamma$ is not free or a surface group,
$\Gamma$ contains infinitely many quasiconvex surface subgroups with 
mutually disjoint boundaries in $\partial\Gamma$ (see Abikoff-Maskit \cite{abikoff-maskit}).

Let $H$ and $J$ be quasiconvex surface subgroups of $\Gamma$ so that $\partial H$ and $\partial J$ are disjoint in $\partial\Gamma$.
Choose $z\in\partial\Gamma-(\partial J\cup\partial H)$. By transversality, both 
$\xi_\rho^1(\partial H)$ and $\xi_\rho^1(\partial J)$ are disjoint from the projective plane $\mathbb P(\xi_\rho^3(z))$,
so are contained in the affine chart $A=\mathbb P(\mathbb R^4-\xi_\rho^3(z))$. 

Since $\rho|_H$ is irreducible, by Proposition \ref{Borel4irreducible},
Corollary \ref{Omegasurface} implies that
$$\Omega_H=\mathbb P\left(\mathbb R^4-\bigcup_{x\in\partial H}\xi_\rho^3(x)\right)$$ 
is a properly convex domain which is $\rho(H)$-invariant and $\xi_\rho^1(\partial H)\subset \partial\Omega_H$.
Let 
$$T_H=\mathbb P\left(\bigcup_{x\in\partial H}\xi_\rho^2(x)\right).$$ By transversality, $T_H$ is a disjoint union of projective lines,
so it is a $S^1$-bundle over the circle $\partial H$. It follows, that $T_H$ is a Klein bottle or a torus. Since the Klein bottle
does not embed in $\mathbb P(\mathbb R^4)$ (see \cite{bredon-wood}), 
$T_H$ is a torus. Notice that $T_H$ separates since $H^2(\mathbb P(\mathbb R^4))=0$.

If $x\in\partial H$, the projective line $\mathbb P(\xi_\rho^2(x))$  intersects the projective plane $\mathbb P(\xi_\rho^3(z))$
in exactly one point, so $C_H=T_H\cap \mathbb P(\xi_\rho^3(z))$ is a simple closed curve. Since, by transversality, $C_H$ is disjoint from
the projective line $\mathbb P(\xi_\rho^2(z))$ in $\mathbb P(\xi_\rho^3(z))$, $C_H$ bounds a disk $D_H$ in the disk
$\mathbb P(\xi_\rho^3(z))\setminus\mathbb P(\xi_\rho^2(z))$.
Notice that $D_H$ is unique, since the other component of $\mathbb P(\xi_\rho^3(z))-C_H$ is an open M\"obius band.

The boundary of the regular neighborhood of $D_H\cup T_H$ 
has a spherical component $S_H$ contained in $A$, which bounds a ball $B_H$ in $A$, since $A$ is irreducible. 
Therefore, $T_H$ bounds an open solid torus $V_H$ which contains $B_H$ and intersects $\mathbb P(\xi^3_\rho(z))$
exactly in $D_H$.
Since $\xi_\rho^1(\partial H)$ is homotopic to $C_H$ it  also bounds a disk in $V_H$. 
However, since $\xi_\rho^1(\partial H)$ is homotopically non-trivial in $T_H$,  it can't also bound a disk in  
$\mathbb P(\mathbb R^4)-V_H$.  (If it bounds a disk both in $V_H$ and in its complement, then the sphere $S_H$ made from
the two disks intersects each projective line in $T_H$ exactly once, which contradicts the fact that every sphere
in $\mathbb P(\mathbb R^4)$ bounds a ball.)

Since $\Omega_H$ is disjoint from $T_H$ and  $\xi_\rho^1(\partial H)$ bounds a disk in $\Omega_H$,
we must have $\Omega_H$ contained in $V_H$.

Now consider the torus
$$T_J=\mathbb P\left(\bigcup_{x\in\partial J}\xi_\rho^2(x)\right),$$ 
simple closed curve $C_J=T_J\cap  \mathbb P(\xi_\rho^3(z))$, 
disk $D_J\subset\mathbb P(\xi_\rho^3(z))$ bounded by $C_J$,
and open solid torus $V_J$ bounded by $T_J$ so that
$$\Omega_J=\mathbb P\left(\mathbb R^4-\bigcup_{x\in\partial J}\xi_\rho^3(x)\right)$$ 
and $D_J$ are both contained in $V_J$.

Since $\xi_\rho^1(\partial J)\subset \Omega_H\subset V_H$ and $T_J$ is disjoint from $T_H=\partial V_H$, by transversality,
$T_J$ is contained in ${\rm int}(V_H)$. Therefore, $C_J$ is contained in $D_H$ which implies that $D_J$ is contained in $D_H$.
So  the regular neighborhood of $D_J\cup T_J$  can be taken to have a spherical component $S_J$  contained in $B_H$, so $S_J$ bounds a ball
$B_J$ contained in $B_H$. Putting this all together, we see that $V_J$ must be contained in $V_H$.
Therefore, $\xi^1_\rho(\partial H)$ is contained in the complement of $V_J$ and hence in the complement of $\Omega_J$.
It follows that 
$$\xi^1_\rho(\partial H)\subset \bigcup_{x\in\partial J}\mathbb P(\xi_\rho^3(x))$$ 
which contradicts transversality. Therefore, $\Gamma$ is either a surface group or a free group.
\eproof

\section{Hyperconvexity}

Labourie \cite{labourie-invent} introduced the theory of Anosov representations in his study of Hitchin representations.
Recall that a representation is {\em $d$-Fuchsian} if it is the composition of a Fuchsian representation of
a surface group into $\mathsf{SL}(2,\mathbb R)$ with the irreducible representation of $\mathsf{SL}(2,\mathbb R)$
into $\mathsf{SL}(d,\mathbb R)$. {\em Hitchin representations} \cite{hitchin} are representations of a surface group into
$\mathsf{SL}(d,\mathbb R)$ which can be continuously deformed to a $d$-Fuchsian representation. Labourie showed
that Hitchin representations are irreducible and Borel Anosov.

Labourie \cite{labourie-invent} and Guichard \cite{guichard-hitchin} proved that a representation
$\rho:\pi_1(S)\to\mathsf{SL}(d,\mathbb R)$ is Hitchin if and only if there exists a hyperconvex limit map,
i.e. a $\rho$-equivariant map $\xi^1_\rho:\partial\pi_1(S)\to \mathbb P(\mathbb R^d)$ so that 
if $\{ x_1,\ldots,x_d\}$ are distinct points in $\partial\Gamma$, then
$\xi_\rho^1(x_1)\oplus\cdots\oplus\xi^1_\rho(x_d)=\mathbb R^d$. Labourie further shows that if
$\rho$ is Hitchin, $n_1,\ldots,n_k\in\mathbb N$, $n_1+\cdots+n_k=d$ and $\{x_1,\ldots,x_k\}$
are distinct points in $\partial\pi_1(S)$, then $\xi_\rho^{n_1}(x_1)\oplus\cdots\oplus \xi_\rho^{n_k}(x_k)=\mathbb R^d$
and that if $\{(y_n,z_n)\}$ is a sequence in $\partial\Gamma\times\partial\Gamma$, with $y_n\ne z_n$ for all $n$, converging to $(x,x)$,
and $p,q,r\in \{1,\ldots,p-1\}$ with $p+q=r$, then $\{\xi_\rho^p(y_n)\oplus\xi^q_\rho(z_n)\}$ converges to $\xi_\rho^r(x)$.

We use these hyperconvexity properties to show that Hitchin representations cannot be extended to representations of larger
groups which are $P_1$-Anosov and $P_2$-Anosov. We consider this to be more evidence for a positive answer to Sambarino's
question.

\begin{proposition}
\label{hitchindontextend}
Suppose $\Gamma$ contains a surface subgroup $\Gamma_0$, $\rho:\Gamma\to \mathsf{SL}(d,\mathbb R)$
is projective Anosov and $\rho|_{\Gamma_0}$ is Hitchin. 
\begin{enumerate}
\item
If $d$ is even, then $\Gamma_0$ has
finite index in $\Gamma$.
\item
If $d$ is odd, and $\rho$ is also $P_2$-Anosov, then $\Gamma_0$ has
finite index in $\Gamma$.
\end{enumerate}
\end{proposition}

\begin{proof} 
First notice that, since $\rho|_{\Gamma_0}$ is projective Anosov, $\Gamma_0$ is quasiconvex in $\Gamma$(\cite[Lemma 2.3]{CLS}),
so $\partial\Gamma_0$ embeds in $\partial\Gamma$.
If $\Gamma_0$ has infinite index, then there exists $z\in\partial\Gamma-\partial\Gamma_0$.
Let $A$ be the affine chart $\mathbb P(\mathbb R^d-\xi_\rho^{d-1}(z))$. By transversality, $\xi_\rho^1(\partial\Gamma_0)\subset A$.
If $d$ is even, this contradicts Lemma 12.3 in Danciger-Gu\'eritaud-Kassel \cite{DGK} which asserts that 
$\xi_\rho^1(\partial\Gamma_0)$ cannot lie in any
affine chart. Therefore, $\Gamma_0$ must have finite index in $\Gamma$.

Now suppose that $d$ is odd, $\rho$ is $P_2$-Anosov and $\Gamma_0$ has infinite index in $\Gamma$.
There exists a continuous map $h:D^2\to A$ and a homeomorphism $g:\partial\Gamma_0\to S^1$ so that
$\xi_\rho^1|_{\partial\Gamma_0}=h\circ g$. Let $V=\xi_\rho^{d-2}(z)^\perp$ and define the continuous map
$$F:D^2\to \mathbb P(V)\cong S^1$$
by letting
$$ F(x)=\left[ \left(h(x)\oplus\xi_\rho^{d-2}(z)\right)\cap V\right].$$

We now claim that $F|_{\partial\Gamma_0}$ is locally injective. If not there exist sequences $\{x_n\}$ and $\{y_n\}$ 
in $\partial\Gamma_0$ so that $x_n\ne y_n$ (for any $n$), $\lim x_n=q=\lim y_n$ and $F(x_n)=F(y_n)$ for all $n$.
Since $\rho|_{\Gamma_0}$ is Hitchin, the sequence 
$\{\xi_\rho^1(x_n)\oplus\xi_\rho^1(y_n)\}$ converges to $\xi_\rho^{2}(q)$.
So, for all $n$, we may choose vectors $u_n$, $v_n$ and $w_n$  in $\xi_\rho^1(x_n)$, $\xi_\rho^1(y_n)$ and $\xi_\rho^{d-2}(z)$
so that $u_n+v_n=w_n$ and $w_n$ is unit length. Up to subsequence, $\{w_n\}$ converges to  a unit vector $w$, but then
$w\in \xi_\rho^2(q)$, since $w_n\in \xi_\rho^1(x_n)\oplus\xi_\rho^1(y_n)$ for all $n$, and $w\in\xi_\rho^{d-2}(z)$, since
$w_n\in\xi_\rho^{d-2}(z)$ for all $n$. However, this violates transversality, since $q\ne z$.
Therefore, $F|_{\partial\Gamma_0}$ is a covering map, which is impossible since $(F|_{\partial\Gamma_0})_*$ is trivial
on $\pi_1$.
\end{proof}

Pozzetti, Sambarino and Wienhard \cite{PSW} recently introduced the notion of $(p,q,r)$-hyperconvex representations which share
specific transversality properties with Hitchin representations. A representation $\rho:\Gamma\to \mathsf{SL}(d,\mathbb R)$ is said
to be $(p,q,r)$-hyperconvex, where  $p+q\le r$,
if $\rho$ is $P_p$, $P_q$ and $P_r$(or $P_{d-r}$)-Anosov and whenever $x,y,z\in\partial\Gamma$ are distinct,
$$\left(\xi_\rho^p(x)\oplus\xi_\rho^q(y)\right) \cap \xi^{d-r}(z)=\{0\}.$$
One may view the following as a generalization of Corollary 6.6 of Pozzetti-Sambarino-Wienhard \cite{PSW} 
which asserts that if $\rho:\Gamma\to \mathsf{SL}(d,\mathbb R)$ is $(1,1,r)$-hyperconvex and $x_0\in\partial\Gamma$, then
there is a continuous injection of $\partial\Gamma-\{x_0\}$ into $\mathbb P(\mathbb R^r)$, see also Lemma 4.10 in Zhang-Zimmer \cite{zhang-zimmer}.
(Pozzetti, Sambarino and Wienhard \cite[Corollary 6.6]{PSW} also applies to representations into $\mathsf{SL}(K)$ where $K$
is any local field.)

\begin{proposition}
Suppose that $\Gamma$ is a torsion-free hyperbolic group and $\rho:\Gamma\to \mathsf{SL}(d,\mathbb R)$ is $P_p$-Anosov.
If there exists a $(d-r)$-plane $V$ such that 
$$V\cap\left(\xi^p(x)\oplus\xi^p(y)\right)=\{0\}$$
if $x,y\in\partial\Gamma,$
then $\Gamma$ has cohomological dimension at most  $r-p+1$.
If $\Gamma$ has cohomological dimension $r-p+1$, then $\partial\Gamma\cong S^{r-p}$ and 
$(r-p,p)$ is either $(1,1)$, $(2,2)$, $(4,4)$ or $(8,8)$.

Moreover, if $\rho$ is $(p,p,r)$-hyperconvex, then $\Gamma$ has cohomological dimension at most $r-p+1$
and if $\Gamma$ has cohomological dimension $r-p+1$, then $\partial\Gamma\cong S^{r-p}$.
\end{proposition}

Notice that if $p\le q$, then $(p,q,r)$-hyperconvex representations are $(p,p,r)$-hyperconvex, so we may conclude that
if $\rho:\Gamma\to \mathsf{SL}(d,\mathbb R)$ is $(p,q,r)$-hyperconvex, then $\Gamma$ has cohomological dimension at most
$r+1-\min\{p,q\}$. Pozzetti, Sambarino and Wienhard \cite[Cor. 7.6]{PSW} observe that if $k\ge 2$ and $\rho:\Gamma\to \mathsf{PO}(d,1)$,
then the $k^{\rm th}$ symmetric power $S^k\rho:\Gamma \to \mathsf{PGL}(S^k(\mathbb R^{d+1}))$ is $(1,1,d)$-hyperconvex,
so one obtains no topological restrictions in the case where $\rho$ is $(1,1,d)$-hyperconvex and $\Gamma$ has maximal
cohomological dimension $d$.

\begin{proof}
Let $p:E\to \partial\Gamma$ be the fibre bundle provided by Lemma \ref{fibrebundle}
where 
$$E=\bigcup_{x\in\partial\Gamma} S(\xi_\rho^p(x))\subset S(\mathbb R^d).$$
If $\partial \Gamma$ has topological dimension $m$, then $E$ has topological dimension $m+p-1$.

Let $\pi:\mathbb R^d\to V$ be orthogonal projection (with respect to some fixed background metric on $\mathbb R^d$).
Let 
$$f:\mathbb R^d-V\to S(V^\perp)\cong S^{r-1}$$
be the continuous map given by
$$f(u)=\frac{u-\pi(u)}{||u-\pi(u)||}.$$
Notice that if $f(u)=f(v)$, then $||u-\pi(u)||u-||v-\pi(v)||v\in V$.
Suppose $u\in S(\xi_\rho^p(x))$ and $v\in S(\xi_\rho^p(y))$ and $f(u)=f(v)$, then
since $\left(\xi_\rho^p(x)\oplus\xi_\rho^p(y)\right)\cap V=\{0\}$, it must be the case that $u=v$
and, since  $\xi_\rho^p(x)\cap\xi_\rho^p(y)=\{0\}$ if $x\ne y$, it must be the case that $x=y$.
Therefore, the restriction $f|_E$  of $f$ to $E$ is injective and hence a topological embedding.
It follows, since $E$ has topological dimension $m+p+1$, that $m+p-1\le r-1$, so $m\le r-p$, which 
implies, by \cite[Cor. 1.4]{bestvina-mess}, that $\Gamma$ has  cohomological dimension at most $r-p+1$.

Suppose that $m=r-p$. We first show that $\partial \Gamma\cong S^{m}=S^{r-p}$. 
Fix $x_0\in\partial\Gamma$. Choose a $(r-p)$-dimensional subspace $W_1\subset\mathbb R^d$ so that $\xi_\rho^p(x_0)\oplus W_1\oplus V=\mathbb R^d$.
Pick a line $L\subset \xi^p_\rho(x_0)$ and let $W=W_1\oplus L$.
There exists a neighborhood $U$ of $x_0$ in $\partial\Gamma$
so that $\left(\xi_\rho^p(x)\oplus V\right)\cap W$ is a line if $x\in U$.  
Consider the continuous map $F:U\to \mathbb P(W)$ given by 
$$F(x)=\left[\left(\xi_\rho^p(x)\oplus V\right)\cap W\right].$$
We argue as above to show that $F$ is injective. If $F(x_1)=[u_1+v_1]$ and $F(x_2)=[u_2+v_2]$, where $u_i\in\xi^p_\rho(x_i)$
and $v_i\in V$, and $F(x_1)=F(x_2)$, we may assume that $u_1+v_1=u_2+v_2$, so 
$u_1-u_2=v_2-v_1\in\left(\xi_\rho^p(x_1)\oplus\xi_\rho^p(x_2)\right)\cap V=\{0\}$.
Therefore, $u_1=u_2$ (and each $u_i$ is non-zero, since $V\cap W=\{0\}$), which implies that $\xi_\rho^p(x_1)\cap\xi_\rho^p(x_2)\ne \{0\}$,
so $x_1=x_2$.
Therefore, $F$ is a topological embedding.
Since $U$ and $\mathbb P(W)$ both have topological
dimension $m=r-p$, $U$ and hence $\partial \Gamma$ contains a manifold point. 
Thus, by Kapovich-Benakli \cite[Th. 4.4]{benakli-kapovich}, $\partial\Gamma\cong S^{m}$.
Moreover, $f|_E$ is one of the four Hopf fibrations, so the only possibilities for $(r-p,p)$ are
$(1,1)$, $(2,2)$, $(4,4)$ and $(8,8)$.

If $\rho$ is $(p,p,r)$-hyperconvex, we choose $x_0\in\partial\Gamma$ and let $V=\xi^{d-r}_\rho(x_0)$ and apply
the same argument to conclude that $E-S(\xi_\rho^p(x_0))$  has topological dimension $m+p-1$ and  
every compact subset embeds in a sphere of dimension $r-1$. We again conclude that $\Gamma$ has  
cohomological dimension at most $r-p+1$.
Similarly, we may show, that if $\Gamma$ has  cohomological dimension $r-p+1$, 
then a neighborhood of a point $z_0\ne x_0\in\partial\Gamma$ embeds in
a projective space of dimension $r-p$, so $\partial\Gamma\cong S^{r-p}$.
\end{proof}

\section{Characterizing Benoist representations by limit maps}

In this section we obtain characterizations of Benoist representations purely in
terms of limit maps. We first work in the setting where 
the domain group doesn't split over a cyclic subgroup.

\begin{theorem}
\label{benoistlimitmap2}
Suppose that $d\ge 4$ and $\Gamma$ is a torsion-free hyperbolic group. 
A representation $\rho:\Gamma \rightarrow \mathsf{SL}(d, \mathbb{R})$  is a Benoist representation
if and only if $\Gamma$ has  cohomological dimension $d-1$,  $\Gamma$ does not split over
a cyclic subgroup and  there is a $\rho$-equivariant continuous non-constant map 
$\xi:\partial\Gamma \rightarrow \mathbb{P}(\mathbb{R}^{d})$.
\end{theorem}

\begin{proof}
If $\rho$ is a Benoist representation then $\Gamma$ has cohomological dimension $d-1$,
$\rho$ is projective Anosov and  $\xi_\rho^1$ is a 
$\rho$-equivariant continuous non-constant map. Moreover, since $\partial\Gamma\cong S^{d-2}$ and
$d\ge 4$, $\Gamma$ does not split over
a cyclic subgroup (\cite[Thm 6.2]{bowditch-cutpoints}). So, the bulk of our work is in establishing the converse.

We  first prove a general result about representations admitting a limit map whose image spans.
We recall that $\rho$ is said to be $P_1$-divergent if whenever  $\{\gamma_n\}$  is a 
sequence of distinct elements in $\Gamma$, then
$$\lim\frac{\sigma_2(\rho(\gamma_n))}{\sigma_1(\rho(\gamma_n))}=0$$
where $\sigma_i(\rho(\gamma_n))$ is the $i^{\rm th}$ singular value of $\rho(\gamma_n)$.

\begin{lemma} 
\label{limitmapprops}
Suppose that $\Gamma$ is a non-elementary hyperbolic group , $d\ge 2$, and
\hbox{$\rho:\Gamma \rightarrow \mathsf{SL}(d, \mathbb{R})$} 
admits a continuous $\rho$-equivariant map $\xi:\partial\Gamma \rightarrow \mathbb{P}(\mathbb{R}^d)$
such that $\xi(\partial\Gamma)$ spans $\mathbb R^d$. Then
the representation $\rho$ is $P_1$-divergent and if $\rho(\gamma)\in\rho(\Gamma)$ is proximal, 
then $\xi(\gamma^+)$ is the attracting eigenline of $\rho(\gamma)$.
In particular, $\rho(\Gamma)$ is discrete and $\rho$ has finite kernel.

Moreover, if in addition, $\rho$ is irreducible, then  $\rho(\Gamma)$ contains a biproximal element.
\end{lemma}

\begin{proof}
If $\rho$ is not $P_1$-divergent, then there exists a sequence $\{\gamma_n\}$ of distinct elements of $\Gamma$
so that  $\lim\frac{\sigma_2(\rho(\gamma_n))}{\sigma_1(\rho(\gamma_n))}=C>0$.
Since $\Gamma$ acts a convergence group on $\partial\Gamma$, we may pass to another subsequence so that
there exist $\eta$ and $\eta'$ such that if $x \neq \eta'$ then $\lim\gamma_n x=\eta$. 
Since $\xi(\partial\Gamma)$ spans $\mathbb R^d$ and $\partial\Gamma$ is perfect,
there exist $x_1,...,x_d \neq \eta'$ such that $\mathbb{R}^d=\oplus \xi(x_i)$. 
Suppose that $\xi(x_i)=[l_{i}e_1]$  for each $i$ and $\xi(\eta)=[l_\eta e_i]$ where $l_i,l_\eta\in\mathsf{O}(d)$.
We write each $\rho(\gamma_n)=k_na_nk_n'$ in the Cartan decomposition where $k_n,k_n'\in \mathsf{O}(d)$ and
$a_n$ is the diagonal matrix so that $a_{ii}=\sigma_i(\rho(\gamma_n))$ for all $n$.
We may pass to another subsequence so that $\{k_n\}$
and $\{k_n'\}$ converge to $k$ and $k'$.
Then, since $\xi$ is $\rho$-equivariant, $\lim\rho(\gamma_n)\xi(x_i)=\xi(\eta)$, so
$\{[k_na_nk_n'l_ie_1]\}$ converges to $[l_\eta e_1]$, which implies that
$\{[a_nk_n'l_ie_1]\}$ converges to $[k^{-1}l_\eta e_1]$ in $\mathbb P(\mathbb R^d)$.
So, perhaps after replacing $l_i$ with $-l_i$, 
$$\lim \frac{a_nk_{n}'l_{i}e_1}{||a_nk_{n}'l_{i}e_1||}=k^{-1}l_{\eta}e_1$$ 
for all $i$.
We pass to a subsequence so that
$$\nu_{i}=\lim \frac{||a_nk_{n}'l_{i}e_1||}{\sigma_1(\rho(\gamma_n))}$$
exists for all $i$.
Then,
\begin{eqnarray*}
 \nu_i\left\langle k^{-1}l_\eta e_1,e_1\right\rangle & = & \lim\left( \frac{||a_nk_{n}'l_{i}e_1||}{\sigma_1(\rho(\gamma_n))}\right)
 \left\langle \frac{a_nk_{n}'l_{i}e_1}{||a_nk_n'l_ie_1||} ,e_1\right\rangle \\
& = & \lim\left\langle a_nk_{n}'l_{i}e_1, \frac{e_1}{\sigma_1(\rho(\gamma_n))}\right\rangle\\
&=& \lim\left\langle k_n'l_ie_i,e_1\right\rangle=\left \langle k'l_{i}e_1,e_1\right \rangle \\
\end{eqnarray*}
A similar calculation, and the fact that $\lim\frac{\sigma_2(\rho(\gamma_n))}{\sigma_1(\rho(\gamma_n))}=C>0$,  yields  
$$\left \langle k'l_{i}e_1,e_2\right \rangle =\frac{\nu_{i}}{C}\left \langle k^{-1}l_{\eta}e_1,e_2 \right \rangle.$$
Since the vectors $\{ k'l_ie_1\}_{i=1,..,d}$ span $\mathbb{R}^d$, there exists $i_0$ such that $\nu_{i_0} \neq 0$. 
Then we further observe that  
$$\left \langle k'l_{i}e_1-\frac{\nu_i}{\nu_{i_0}}k'l_{i_0}e_1,e_1\right \rangle =
\left \langle \nu_i k^{-1}l_\eta e_1-\frac{\nu_i}{\nu_{i_0}}\nu_{i_0}k^{-1}l_\eta e_1,e_1\right \rangle=0$$
and similarly that
$$\left \langle k'l_{i}e_1-\frac{\nu_i}{\nu_{i_0}}k'l_{i_0}e_1,e_2\right \rangle=0.$$ 
Therefore, each $k'l_{i}e_1$ lies in the subspace of $\mathbb R^d$
spanned by  $k'l_{i_0}e_1$ and $e_1^{\perp}\cap e_2^{\perp}$, which has codimension at least one.
This however contradicts the fact that $\{ k'l_ie_1\}_{i=1,..,d}$ spans $\mathbb{R}^d$.
Therefore, $\rho$ is $P_1$-divergent. Notice that if $\rho(\gamma)\in\rho(\Gamma)$ is proximal, then there
exists $x\in\partial\Gamma$ so that $\xi(x)$ does not lie in the repelling hyperplane of $\rho(\gamma)$,
so $\rho(\gamma^n)(x)$ converges to the attracting eigenline of $\rho(\gamma)$. Since $\xi$ is $\rho$-equivariant
and $\lim\gamma^n(x)=\gamma^+$, $\xi(\gamma^+)$ is the attracting eigenline of $\rho(\gamma)$.

We now assume that $\rho$ is also irreducible. Theorem \ref{finitesubset} provides
a finite subset $A$ of $\Gamma$ and $M>0$ so
that if $\gamma\in\Gamma$, then there exists $\alpha\in A$ so that
$$|\log\lambda_i(\rho(\gamma \alpha))-\log\sigma_i(\rho(\gamma))|\le M$$
for all $i$. Let $\{\gamma_n\}$ be an infinite sequence of distinct elements in $\Gamma$ and let $\{\alpha_n\}$ be the
associated sequence of elements of $A$.
Since $\rho$ is $P_1$-divergent,  $\frac{\lambda_1(\rho(\gamma_n\alpha_n))}{\lambda_2(\rho(\gamma_n\alpha_n))}\to\infty$
and $\frac{\lambda_{d-1}(\rho(\gamma_n\alpha_n))}{\lambda_d(\rho(\gamma_n\alpha_n))}\to\infty$,
so $\rho(\gamma_n\alpha_n)$ is biproximal  for all large enough $n$.
\end{proof}

We now complete the proof in the case where $\rho$ is irreducible.

\begin{proposition}
\label{irredcase}
Suppose that
$\Gamma$ is a torsion-free hyperbolic group of cohomological dimension $m\ge d- 1$ which does not admit a cyclic splitting.
If $\rho:\Gamma \rightarrow \mathsf{SL}(d, \mathbb{R})$ is irreducible and there exists a $\rho$-equivariant continuous map 
$\xi:\partial\Gamma \rightarrow \mathbb{P}(\mathbb{R}^{d})$, then $m=d-1$ and
$\rho$ is a Benoist representation.
\end{proposition}

\begin{proof}
Since $\rho$ is irreducible, $\xi(\partial\Gamma)$ spans $\mathbb R^d$, since $\rho(\Gamma)$
preserves the space spanned by $\xi(\partial\Gamma)$. Lemma \ref{limitmapprops} allows us to choose
$\gamma_0\in\Gamma$ so that $\rho(\gamma_0)$ is biproximal. 
We may assume that the attracting eigenlines of $\rho(\gamma_0)$ and $\rho(\gamma_0^{-1})$ are 
$<e_1>$ and  $<e_{d}>$ respectively and the corresponding attracting hyperplanes 
are $e_{d}^{\perp}$ and $e_{1}^{\perp}$. 
In particular, $\xi(\gamma_0^{+})=[e_1]$ and $\xi(\gamma_0^{-})=[e_{d}]$. 
Suppose that $x \in \partial\Gamma- \{\gamma_0^{+},\gamma_0^{-}\}$. 
Since $\lim \gamma_0^n(x)=\gamma_0^+$ and $\lim \gamma_0^{-n}(x)=\gamma_0^-$, $\xi(x)$ cannot lie in either
$\mathbb{P}(e_1^{\perp})$ or  $\mathbb{P}(e_{d}^{\perp})$. 
Since  the group $\Gamma$ does not split over a cyclic subgroup, the set \hbox{$\partial\Gamma -\{\gamma_0^{\pm}\}$}
is connected (see Bowditch \cite[Thm 6.2]{bowditch-cutpoints}),
so we may assume that $\xi(\partial\Gamma -\{\gamma_0^{\pm}\})$ is contained in the connected component 
$ \{[1:x_2:...:x_{d}:]:x_d>0\}$ of $\mathbb{P}(\mathbb{R}^{d})-\mathbb{P}(e_1^{\perp})\cup \mathbb{P}(e_d^{\perp})$. 
It follows that $\xi(\partial\Gamma)$ lies in the affine chart 
$\mathbb{P}(\mathbb{R}^{d})-\mathbb{P}(V)$ where \hbox{$V=\{(x_1,...,x_d) \in \mathbb{R}^{n+1}:x_1=-x_d\}$}. 
Lemma \ref{Zimmerconnected} then implies that 
$\rho(\Gamma)$ preserves a properly convex domain $\Omega$ in $\mathbb P(\mathbb R^d)$. 
Since $\rho(\Gamma)$ is \hbox{$P_1$-divergent}, it is discrete and faithful, so it must acts properly  discontinuously on $\Omega$
(see \cite[Fact 2.10]{benoist-divisible3}).
Since $\rho(\Gamma)$ has cohomological dimension $m\ge d-1$, it must have compact quotient.
Hence, by Benoist \cite[Thm 1.1]{benoist-divisible1}, $\Omega$ is strictly convex, so $\rho$ is a Benoist representation
and $m=d-1$.
\end{proof}

It remains to rule out the case where $\rho$ is reducible. We first deal with the case where $\xi(\partial\Gamma)$ spans $\mathbb R^d$.

\begin{proposition}
\label{redcase}
Suppose that 
$\Gamma$ is a torsion-free hyperbolic group of cohomological dimension $m\ge d-1\ge 2$ which does not admit a cyclic splitting.
If $\rho:\Gamma \rightarrow \mathsf{SL}(d, \mathbb{R})$ is a representation and there exists a $\rho$-equivariant continuous non-constant map 
$\xi:\partial\Gamma \rightarrow \mathbb{P}(\mathbb{R}^{d})$ so that $\xi(\partial\Gamma)$ spans $\mathbb R^d$,
then $\rho$ is irredicuble.
\end{proposition}

\begin{proof}
If $\rho$ is reducible, one may conjugate it to have the form 
$$
{\begin{bmatrix}
\rho_1 & \ast & \ast  & \ast \\ 
 0& \rho_2 &\ast &\ast\\ 
 0& 0 & \ddots & \ast&\\ 
 0& 0 & 0 & \rho_k
\end{bmatrix}}$$ 
where $k\ge 2$, each $\rho_{i}:\Gamma \rightarrow \textup{GL}(V_i)$ is an $d_i$-dimensional irreducible representation and 
\hbox{$\mathbb R^d=\oplus_{i=1}^k V_i$}. Notice that if $x \in \partial\Gamma$ and $\xi(x)$  lies in $\hat V=\mathbb{R}^{d-d_k}\times\{0\}^{d_k}$
then, since $\Gamma$ acts minimally on $\partial\Gamma$ and $\rho(\Gamma)$ preserves $\hat V$, $\xi(\partial\Gamma)$
would be contained in the proper subspace $\hat V$, which would contradict our assumption that $\xi(\partial\Gamma)$ spans $\mathbb R^d$.
It follows that there exists a $\rho_k$-equivariant map 
$\xi_{k}:\partial\Gamma \rightarrow \mathbb{P}(\mathbb{R}^{d_k})$, obtained by letting $\xi_k(x)$ denote the orthogonal projection of $\xi(x)$
onto  $V_k$.
Notice that $\xi_k(\partial\Gamma)$ spans $V_k$, since $\xi(\partial\Gamma)$ spans $\mathbb R^d$.
Proposition \ref{irredcase}, applied to the representation $\rho_k$,
implies that $\Gamma$ has cohomological dimension $d_k-1$, which is a contradiction.
\end{proof}

In the final case of the proof of Theorem \ref{benoistlimitmap2}, $W=\langle \xi(\partial\Gamma) \rangle$ is a proper subspace of $\mathbb{R}^{d}$.
Let $\pi_W:\mathsf{GL}(W)\to\mathsf{SL}^\pm(W)$ be the obvious projection map. Consider
$\hat\rho=\pi_W\circ\rho|_W:\Gamma\to \mathsf{SL}^{\pm}(W)$ and the non-constant $\hat\rho$-equivariant map $\hat\xi:\partial\Gamma\to \mathbb P(W)$
(which is simply $\xi$ with the range regarded as $\mathbb P(W)$). Since $\xi$ is non-constant, $W$ has dimension at least 2.
If $W$ has dimension 2, then, by Lemma \ref{limitmapprops}, $\rho$ is discrete and
faithful, which implies that $\Gamma$ is a free group or surface group, contradicting our assumptions on $\Gamma$.
If $W$ has dimension at least 3, then Proposition \ref{redcase} implies that $\hat\rho$ is irreducible, while
Proposition \ref{irredcase} provides a contradiction in this case.
\end{proof}

We next observe that if $\rho:\Gamma\to \mathsf{SL}(d,\mathbb R)$ has a non-constant spanning limit map, then
the restriction to the boundary of any non-abelian quasiconvex subgroup is also non-constant.

\begin{lemma} 
\label{factornonconstant}
Suppose that  $\Gamma$ is a torsion-free hyperbolic group  and $\Gamma_0$ is a non-abelian quasiconvex subgroup of $\Gamma$. 
If $\rho:\Gamma \rightarrow \textup{SL}(d, \mathbb{R})$  admits a continuous $\rho$-equivariant map 
$\xi:\partial \Gamma \rightarrow \mathbb{P}(\mathbb{R}^d)$ so that $\xi(\partial\Gamma)$ spans $\mathbb R^d$,
then the restriction of $\xi$ to $\partial \Gamma_0$ is non-constant. 
\end{lemma}

\begin{proof}  
Lemma \ref{limitmapprops} implies that  $\rho$ is discrete and faithful. 
Suppose that $\xi$ is constant on $\partial \Gamma_0$. By conjugating, we
may assume $\xi(\partial\Gamma_0)=\{[e_1]\}$.
Then  $\rho|_{\Gamma_0}$ has the form 
$$\rho(\gamma)=\begin{bmatrix}
\varepsilon(\gamma) & u(\gamma) \\ 
0 & |\varepsilon(\gamma)|^{-1/(d-1)}\rho_0(\gamma)
\end{bmatrix}$$ 
for some homomorphism $\varepsilon:\Gamma\to\mathbb R^*$ and some representation 
$\rho_0:\Gamma_0 \rightarrow \mathsf{SL}(d-1, \mathbb{R})$. 
Notice that the representation of $\hat\rho:\Gamma_0\to\mathsf{SL}(d,\mathbb R)$ given by 
$$\hat\rho(\gamma)=\begin{bmatrix} \varepsilon(\gamma) &0 \\ 
 0&  |\varepsilon(\gamma)|^{-1/(d-1)} \rho_0(\gamma)
\end{bmatrix}$$
is the limit of the discrete faithful representations 
$\{Q_n^{-1}\circ \rho|_{\Gamma_0}\circ Q_n\}$, where $Q_n$ is a diagonal matrix with $a_{11}=n$ and all other diagonal entries
equal to $1$, so $\hat\rho$ is discrete and faithful (see Kapovich \cite[Thm. 8.4]{kapovich-book}).

We next show that if $\gamma\in\Gamma_0$ and $\varepsilon(\gamma)=1$, then $\lambda_i(\hat\rho(\gamma))=1$ for all $i$.
If not, consider the Jordan normal formal  for $\rho_0(\gamma)$, regarded as a matrix in $\mathsf{SL}(d-1,\mathbb C)$, i.e.
$$
\rho_0(\gamma)=P\begin{bmatrix}
J_{q_1, k_1} &  & \\ 
 &  \ddots & \\ 
 &  & J_{q_r, k_r}
\end{bmatrix}P^{-1}$$ 
where $P \in \textup{SL}(d-1, \mathbb{C})$ and $J_{q, k}$ is the $k$-dimensional Jordan block with the value $q\in\mathbb C$ 
along the diagonal.
We may assume that $|q_1| \geqslant\cdots\geqslant |q_r|$ and that if $|q_i|=|q_{i+1}|$, then $k_i\ge k_{i+1}$.
Notice that, if $n$ is sufficiently large, the co-efficient of $J_{q,k}^n$ with largest modulus has modulus exactly $\binom{n}{k-1}|q|^{n-k+1}$
It follows that there exists $C>1$ so that 
$$\frac{1}{C}\binom{n}{k_1-1}|q_1|^{n-k_1+1}\le\left \| \rho_0(\gamma^n) \right \|\le C\binom{n}{k_1-1}|q_1|^{n-k_1+1}$$
for all $n \in\mathbb N$. Therefore, $\left \{ \left(\binom{n}{k_1-1}|q_1|^{n-k_1+1} \right)^{-1}\hat\rho(\gamma^n)\right \}$ 
has a subsequence which converges to a non zero matrix $A_{\infty}$. 
One may then show that if $w\in\mathbb R^d$, does not lie in 
the kernel $K$ of $A_\infty$,
then $\{[\rho(\gamma)^n(w)]\}$ does not converge to $[e_1]$. Since $\xi(\partial\Gamma)$ spans $\mathbb R^d$ and $K$ is
a proper subspace of $\mathbb R^d$, there exists
$x\in\partial\Gamma-\partial\Gamma_0$ so that $\xi(x)$ does not lie in $K$. Since $\xi$ is $\rho$-equivariant,
$\{\rho(\gamma)^n(\xi(x))\}$ must converge to $\xi(\gamma^+)=[e_1]$, which is a contradiction.

Notice that if $N$ is the commutator subgroup of $\Gamma_0$, then   $\varepsilon(N)=\{1\}$. Since $\Gamma_0$ is a non-abelian
torsion-free hyperbolic group, $N$ contains a free subgroup $\Delta$ of rank 2.
Let $\psi=\hat\rho|_{\Delta}^{ss}$ be a semisimplification of $\hat\rho|_{\Delta}$. 
Since $\psi$ is a limit of conjugates of $\hat\rho|_\Delta$ and
$\hat\rho|_\Delta$ is discrete and faithful, $\psi$ is also discrete and faithful \cite[Thm 8.4]{kapovich-book}.
Since $\log\lambda_i(\psi(\gamma))=\log\lambda_i(\hat\rho(\gamma))=0$ for all $\gamma\in\Delta$ and all $i$,
Theorem \ref{finitesubset}  guarantees that there exists $M$ so that 
$||\log\sigma_i(\psi(\gamma)||\le M$ for all $\gamma\in \Delta$ and all $i$.
Therefore, $\psi(\Delta)$ is bounded which contradicts the fact that $\psi$ is discrete and faithful and that $\Delta$
is infinite.
\end{proof}

The work of Louder-Touikan \cite{louder-touikan} allows us to find cohomologically large quasiconvex subgroups
which do not split.

\begin{proposition}
If $\Gamma$ is a torsion-free hyperbolic group of cohomological dimension $m\ge 3$ which splits over
a cyclic subgroup, then $\Gamma$ contains an infinite index, quasiconvex subgroup  of cohomological
dimension $m$ which does not split over a cyclic subgroup.
\end{proposition}

\begin{proof}
One first considers a maximal splitting of $\Gamma$ along cyclic subgroups. 
One of the factors, say $\Delta$ has cohomological dimension $m$ (see Bieri  \cite[Cor. 4.1]{bieri-hnn} and Swan \cite[Thm. 2.3]{swan}).
A result of Bowditch \cite[Prop. 1.2]{bowditch-cutpoints}, implies that $\Delta$ is a quasiconvex subgroup of $\Gamma$.
If $\Delta$ itself splits along a cyclic subgroup, we consider a maximal splitting of $\Delta$ along cyclic subgroups.
We then again find a factor $\Delta_1$ of this decomposition which has  cohomological dimension $m$ and is quasiconvex
in $\Delta$, hence in $\Gamma$.
Louder and Touikan \cite[Cor. 2.7]{louder-touikan}
implies that this process terminates after finitely many steps, so
one obtains the desired quasiconvex subgroup of cohomological dimension $m$.
\end{proof}

We now combine the above results to establish Theorem \ref{benoistlimitmap}.

\medskip\noindent
{\bf Theorem \ref{benoistlimitmap}:} {\em
If $d\ge 4$ and $\Gamma$ is a torsion-free hyperbolic group,  a representation 
\hbox{$\rho:\Gamma \rightarrow \mathsf{SL}(d, \mathbb{R})$}  is a Benoist representation
if and only if $\Gamma$ has cohomological dimension $d-1$
and there is a non-constant $\rho$-equivariant continuous map $\xi:\partial \Gamma \rightarrow \mathbb{P}(\mathbb R^d)$.
}

\begin{proof}
If $\rho$ is a Benoist representation then $\Gamma$ has cohomological dimension $d-1$ and $\xi_\rho^1$ is a continuous,
non-constant $\rho$-equivariant map. 

Now suppose that $\Gamma$ has cohomological dimension $d-1$
and there is a non-constant \hbox{$\rho$-equivariant} map $\xi:\partial \Gamma \rightarrow \mathbb{P}(\mathbb R^d)$.
If $\Gamma$ does not split over a cyclic subgroup, then Theorem \ref{benoistlimitmap2} implies that $\rho$ is a Benoist
representation. If $\Gamma$ does split over a cyclic group, let $\Gamma_1$ be an infinite index, quasiconvex subgroup of $\Gamma$
of cohomological dimension $d-1$ which does not split over a cyclic subgroup.  

We next observe that $\xi(\partial\Gamma)$ must span $\mathbb R^d$. If it doesn't, let
$W$ be the subspace spanned by $\xi(\partial\Gamma)$. We obtain a representation
$\hat\rho:\Gamma\to \mathsf{SL}^\pm(W)$, given by $\pi_W\circ\rho|_W$ and a continuous $\hat\rho$-equivariant map $\hat\xi:\partial\Gamma\to\mathbb P(W)$,
which is simply $\xi$ with the range regarded as $\mathbb P(W)$,
so that $\hat\xi(\partial\Gamma)$ spans $W$. There exists a subgroup $\Gamma_2$ of index at most two in $\Gamma_1$,
so that $\hat\rho(\Gamma_2)$ lies in $\mathsf{SL}(W)$. Notice that $\Gamma_2$ also has cohomological dimension $d-1$
and does not split over a cyclic subgroup.
By Proposition \ref{factornonconstant},  $\hat\xi|_{\partial\Gamma_2}$ is non-constant, 
so Propositions  \ref{irredcase} and \ref{redcase} imply that $\hat\rho|_{\Gamma_2}$ is a Benoist representation
and that $W$ has dimension $d$, which is a contradiction.

Since $\xi(\partial\Gamma)$ spans $\mathbb R^d$,
Proposition \ref{factornonconstant} implies that $\xi|_{\partial\Gamma_1}$ is non-constant,
so Theorem \ref{benoistlimitmap2} implies that $\rho_1=\rho|_{\Gamma_1}$ is a Benoist representation.
Therefore, $\rho(\Gamma_1)$ acts properly discontinuously and cocompactly on
$$\Omega=\mathbb P\left(\mathbb R^d-\bigcup_{x\in\partial\Gamma_1}\xi_{\rho_1}^{d-1}(x)\right)$$
where  $\xi_{\rho_1}^{d-1}$ is the limit map for $\rho_1$. 
Moreover, $\xi(\partial\Gamma_1)=\partial\Omega$.

Suppose that $\alpha\in\Gamma-\Gamma_1$ and $\rho(\alpha)$ is biproximal. Let
$V(\alpha)$ be the repelling hyperplane of $\rho(\alpha)$. 
Since $\xi$ is equivariant, if $x\in\partial\Omega$, then
$\{\rho(\alpha^n)(x)\}$ converges to $\xi(\alpha^+)$. Therefore,
$V(\alpha)$ is disjoint from $\partial\Omega$.
It follows that $\mathbb P(\mathbb R^d)-\Omega$ is the closure of the set of repelling hyperplanes of biproximal elements of
$\rho(\Gamma)$. Therefore, the complement of  $\Omega$, and hence $\Omega$ itself,  is invariant under the
full group $\rho(\Gamma)$. Since $\rho(\Gamma)$ is discrete, by Lemma \ref{limitmapprops}, and 
$\rho(\Gamma_1)$ acts cocompactly on $\Omega$,
$\rho(\Gamma_1)$ must have finite index in $\rho(\Gamma)$ which contradicts the fact that $\rho$ is faithful.
\end{proof}

\medskip\noindent
{\bf Remarks:} (1) In the 3-dimensional case, one may show that if $\Gamma$ is a torsion-free hyperbolic group
and $\rho:\Gamma\to \mathsf{SL}(3,\mathbb R)$ admits a non-constant continuous $\rho$-equivariant map
\hbox{$\xi:\partial\Gamma\to \mathbb P(\mathbb R^3)$}, then $\Gamma$ is a surface group or a free group.
If the space $W$ spanned by $\xi(\partial\Gamma)$ is $2$-dimensional, then it follows from Lemma \ref{limitmapprops}
that $\rho|_W:\Gamma\to\mathsf{SL}(W)$ is discrete and faithful, so $\Gamma$ is a surface group or a free group.
Thus, we may assume that $\xi(\partial\Gamma)$ spans $\mathbb R^3$, so, again by Lemma \ref{limitmapprops},
$\rho$ is discrete and faithful and $\rho(\Gamma)$ contains a biproximal element.  

Corollary B of Wilton \cite{wilton} gives that if $\Gamma$ is not free or a surface group, 
then $\Gamma$ contains either an infinite index quasiconvex surface subgroup 
or a quasiconvex group which does not split over a cyclic subgroup. If $\Gamma$ contains a quasiconvex subgroup $\Delta$ which
does not split over a cyclic subgroup, then Propositions \ref{irredcase} and \ref{redcase} imply that $\rho|_{\Delta}$ is a Benoist
representation, and thus, by Theorem \ref{sl3rcase}, $\Delta$ is a surface group, which is a contradiction. If $\Gamma$ contains
a quasiconvex surface subgroup $\Gamma_0$ of infinite index, then, by Lemma \ref{factornonconstant},
$\xi|_{\partial\Gamma_0}$ is non-constant. There exists a biproximal element $\rho(\alpha)\in\rho(\Gamma)-\rho(\Gamma_0)$,
and, since $\xi$ is $\rho$-equivariant, $\xi(\partial\Gamma_0)$ cannot intersect $\mathbb P(V(\alpha))$
where $V(\alpha)$ is the repelling hyperplane of $\rho(\alpha)$,
which implies that $\rho(\partial\Gamma_0)$ lies in an affine chart. If the span $W_0$ of $\xi(\partial\Gamma_0)$ is a proper subspace
of $\mathbb R^3$, then, since $(\rho|_{\Gamma_0})|_{W_0}:\Gamma_0\to\mathsf{SL}(W_0)$ is discrete and faithful, 
$\xi(\partial\Gamma_0)=\mathbb P(W_0)$ intersects $\mathbb P(V(\alpha))$, which is a contradiction.  We then
argue, just as in the proof of Proposition \ref{irredcase}, that $\rho|_{\Gamma_0}$ is a Benoist representation. We further argue,
as in the proof of Theorem \ref{benoistlimitmap}, that this is impossible if $\Gamma_0$ has infinite index in $\Gamma$.

(2) One may use similar techniques to show that in three of the four exceptional cases in Theorem \ref{cohomdimd-k}
one does not even have a non-constant limit map into $\mathbb P(\mathbb R^d)$. More precisely, if 
$k$ is $2$, $4$ or $8$, $\Gamma$ is torsion-free hyperbolic group, $\partial\Gamma\cong S^{k}$ and
$\rho:\Gamma \rightarrow \textup{SL}(2k, \mathbb{R})$ is $P_k$-Anosov, then
there does not exist a non-constant, continuous $\rho$-equivariant map
$\xi:\partial\Gamma\to \mathbb P(\mathbb R^{2k})$. 

Suppose that  $\xi:\partial\Gamma\to \mathbb P(\mathbb R^{2k})$ is a non-constant, continuous $\rho$-equivariant map.
If $\rho$ is not irreducible, then let $W$ be a proper $\rho(\Gamma)$-invariant subspace of $\mathbb R^d$.
One may show that the dimension $W\cap \xi^k(x)$ is constant, say $r$, over $\partial \Gamma$.
Let $p=p_{\rho,k}:E \to \partial \Gamma$ be the fibre bundle given by Lemma \ref{fibrebundle}. Recall, from the proof of
Theorem \ref{cohomdimd-k}, that in these exceptional cases $E=S(\mathbb R^{2k})$.
The restriction \hbox{$q=p|_{S(W)}:S(W)\to\partial\Gamma$} is then a fibre bundle with base space $S^{k}$, fibres homeomorphic
to $S^{r-1}$ and total space $S(W)$. However, since ${\rm dim}(W)<2k$, this is impossible, by the classification
of sphere fibrations (\cite{adams}). So, we may assume that $\rho$ is irreducible.
Lemma \ref{limitmapprops} then implies that  $\rho$ is $P_1$-divergent and that 
there exists a biproximal element $\rho(\gamma)\in\rho(\Gamma)$ and $\xi(\gamma^+)$ is the
attracting eigenline of $\rho(\gamma)$. Therefore, since $\rho$ is \hbox{$P_k$-Anosov,}
\hbox{$\xi(\gamma^+)\subset \xi_\rho^k(\gamma^+)$}.
Since $\xi$ and $\xi_\rho^k$ are both $\rho$-equivariant and $\Gamma$ acts minimally on $\partial\Gamma$, we see that
$\xi(x)\subset\xi_\rho^k(x)$ for all $x\in\partial\Gamma$. Therefore, $\xi$ lifts to a section of the spherical fibration $p$, which
we have already seen is impossible.

\section{Examples and Questions}
\label{examples}

In this section, we collect examples related to our results and discuss questions that arise.

It is natural to ask when the cohomological dimension bounds provided by Theorem \ref{cohomdimd-k} are sharp.
Cocompact lattices in $\mathsf{SO}(d,1)\subset \mathsf{SL}(d+1,\mathbb R)$ have cohomological dimension $d$
and the inclusion map is a projective Anosov, so our results are sharp whenever $k=1$.
Cocompact lattices in $\mathsf{SU}(d,1)\subset \mathsf{SL}(2d+2,\mathbb R)$ have cohomological dimension $2d$ and
are $P_2$-Anosov, so our results are sharp when  $k=2$ and $d>4$ is even. 
Similarly, cocompact lattices in $\mathsf{Sp}(n,1)\subset \mathsf{SL}(4n+4,\mathbb R)$
demonstrate sharpness when $k=4$ and $d=4n+4>8$

We also note that all the exceptional cases in part (2) of Theorem \ref{cohomdimd-k} occur. 
If \hbox{$\rho:\Gamma\to \mathsf{SL}(2,\mathbb R)$} is Fuchsian and $\Gamma$ is a surface group,
then $\Gamma$ has cohomological dimension 2 and $\rho$ is projective Anosov. If $\Gamma$ is a cocompact lattice in
$\mathsf{SL}(2,\mathbb C)\subset\mathsf{SL}(4,\mathbb R)$, then $\Gamma$ has cohomological dimension 3 and the inclusion
map is $P_2$-Anosov. Similarly, cocompact lattices in $\mathsf{SL}(2,\mathcal{Q})\subset \mathsf{SL}(8,\mathbb R)$ and
$\mathsf{SL}(2,\mathcal{O})\subset \mathsf{SL}(16,\mathbb R)$ have cohomological dimension $5$ and $9$ and are $P_4$-Anosov
and $P_8$-Anosov, respectively, where $\mathcal{Q}$ is the quaternions and $\mathcal{O}$ is the octonions. 
(Notice that $\mathsf{PSL}(2,\mathcal{Q})$ may be identified with $\mathsf{SO}_0(5,1)$, in such a way that $\partial \mathbb H^5\cong S^4$
is identified with $\mathcal{Q}\mathbb P^1$, hence if $\Gamma$ is a cocompact lattice in $\mathsf{SL}(2,\mathcal{Q})$, then $\Gamma$ is
hyperbolic, $\partial\Gamma\cong S^4$ and there is an equivariant homeomorphism from $\partial\Gamma$ to
$\mathcal{Q}\mathbb P^1\subset {\rm Gr}_4(\mathbb R^8)$. Similarly, $\mathsf{PSL}(2,\mathcal{O})$ is identified with
$\mathsf{SO}_0(9,1)$ and one may make a similar analysis. See Baez \cite{baez} for more details.)

\medskip\noindent
{\bf Question 1:}
{\em For what values of $d$ and $k$ are the estimates in Theorem \ref{cohomdimd-k} sharp?}

\medskip

If $\Gamma$ is a convex cocompact subgroup of $\mathsf{PSL}(2,\mathbb{C})\cong \mathsf{SO}_0(3,1)$, the inclusion map lifts
to a representation $\rho:\Gamma\to \mathsf{SL}(2,\mathbb C)\subset\mathsf{SL}(4,\mathbb R)$. In light
of Theorem \ref{sl4rcase} it is natural to ask:

\medskip\noindent
{\bf Question 2:}
{\em If $\Gamma$ is a torsion-free hyperbolic group and $\rho:\Gamma\to \mathsf{SL}(4,\mathbb R)$ is $P_2$-Anosov,
must $\Gamma$ be isomorphic to a convex cocompact subgroup of $\mathsf{PO}(3,1)$?}

\medskip

We know of no examples of Borel Anosov representations of surface groups in even dimensions which are not Hitchin.
Proposition \ref{Borel4irreducible} assures us that every Borel Anosov representation of a surface group into $\mathsf{SL}(4,\mathbb R)$
is irreducible. Together, they suggest the following ambitious question.

\medskip\noindent
{\bf Question 3:} {\em Is every Borel Anosov representation of a surface group into $\mathsf{SL}(4,\mathbb R)$ Hitchin?}

\medskip

Notice that Danciger and Zhang \cite[Thm. 1.3]{danciger-zhang} proved that Hitchin representations into $\mathsf{SO}(n,n)$
are not $P_n$-Anosov, if you regard them as representations into $\mathsf{SL}(2n,\mathbb R)$.

\medskip

We characterize Borel Anosov subgroups in dimensions 3 and 4. We note that it is easy to show that
a cocompact lattice $\Gamma$ in $\mathsf{Sp}(n,1)$ does not admit a Borel Anosov representation 
into $\mathsf{SL}(d,\mathbb R)$ for any $d$.
Suppose that $\rho: \Gamma \rightarrow \mathsf{SL}(d, \mathbb{R})$ is Borel Anosov. By Corlette \cite{corlette} and
Gromov-Schoen's \cite{gromov-schoen}  superrigidity theorem, see also \cite{fisher-hitchman}, there exists 
$\rho_1:\mathsf{Sp}(n,1) \rightarrow \mathsf{SL}(d, \mathbb{R})$ and $\rho_2: \Gamma \rightarrow \mathsf{SL}(d, \mathbb{R})$ 
with compact closure so that $\rho=(\rho_1|_{\Gamma})\rho_2$ and $\rho_{1}|_\Gamma$ and  $\rho_2$ commute. 
Since $\rho$ does not have compact closure, the representation $\rho_1$ has discrete kernel (in fact central).
Let $\gamma\in\Gamma$ have infinite order. Then, the centralizer $Z$ of $\gamma$ in $\mathsf{Sp}(n,1)$ is non-abelian,
which implies that the centralizer of $\rho(\gamma)$ in $\mathsf{SL}(d,\mathbb R)$ contains $\rho_1(Z)$ and is hence
non-abelian. However, $\rho(\gamma)$ is diagonalizable with distinct eigenvalues, hence has abelian centralizer,
so we have arrived at a contradiction.

\medskip\noindent
{\bf Question 4:} {\em What other classes of hyperbolic groups can be shown not to admit Borel Anosov representations in any dimension?}

\medskip

It is expected that not all linear hyperbolic group admit linear Anosov representations, but we know of
no explicit examples. See also the discussion in Kassel \cite[Section 8]{kassel}.

\medskip\noindent
{\bf Question 5:} {\em Can one exhibit explicit examples of linear hyperbolic groups which do not admit
Anosov representations into $\mathsf{SL}(d,\mathbb R)$ for any $d$?}

\medskip

One can exhibit a sequence $\{\Gamma_k\}$ of hyperbolic groups so that each $\Gamma_k$  has virtual
cohomological dimension 2, admits a faithful representation into $\mathsf{SL}(2k,\mathbb R)$, but admits
no projective Anosov representation into $\mathsf{SL}(2k+1,\mathbb R)$.
Let $\Gamma_1=\pi_1(S)*\mathbb Z$. In general, we define $\Gamma_{k}=(\Gamma_{k-1}\oplus \mathbb{Z}_3)\ast \mathbb{Z}$
and notice that $\Gamma_k$ can be realized as a subgroup of  
$\hat\Gamma_k=(\Gamma_{k-1} \ast \mathbb{Z})\oplus (\mathbb{Z}_3 \ast \mathbb{Z})$.
It is not difficult to check that there is a faithful representation $\rho_1:\Gamma_1\to\mathsf{SL}(2,\mathbb R)$.
Theorem \ref{sl3rcase} implies that $\Gamma_1$ does not admit a projective Anosov representation
into $\mathsf{SL}(3,\mathbb R)$.  Since $\Gamma_{k-1}$ contains a subgroup isomorphic to
$\Gamma_{k-1}*\mathbb Z$, there exists a faithful representation
\hbox{$\hat\rho_k:\Gamma_{k-1}*\mathbb Z\to\mathsf{SL}(2k-2,\mathbb R)$}. If $\sigma:\mathbb Z_3 *\mathbb Z\to \mathsf{SL}(2,\mathbb R)$ is a faithful representation
and $\pi_i$ is the projection of $\hat\Gamma_k$ onto the $i^{\rm th}$ summand, then
$(\rho_{k-1}\circ\pi_1)\oplus (\sigma\circ\pi_2)$ is a faithful representation of $\hat\Gamma_k$ into $\mathsf{SL}(2k,\mathbb R)$, 
which restricts to a faithful representation $\rho_k:\Gamma_k\to\mathsf{SL}(2k,\mathbb R)$.
Suppose $\rho:\Gamma_k \rightarrow \textup{SL}(2k+1, \mathbb{R})$ is projective Anosov. 
Let $c$ be the generator of $\mathbb Z_3$ in the first factor of $\Gamma_k$.  
Since the element $c$ fixes $\partial \Gamma_{k-1}$ pointwise, $V=\langle \xi_\rho^1(\partial \Gamma_{k-1}) \rangle$ is contained 
in the kernel of $\rho(c)-I$, which has dimension at most $2k-2$.  
(Notice that $\rho(c)\ne I$, since $\rho$ has finite kernel, and $c$ is not contained in a  finite normal subgroup of $\Gamma_k$.)
However, the restriction $\rho|_{V}: \Gamma_{k-1} \rightarrow \textup{GL}(V)$ would then be projective Anosov, 
which is impossible by our inductive assumption.

\medskip

It is a consequence of the Geometrization Theorem that any hyperbolic group which admits a discrete faithful
representation into $\mathsf{PO}(3,1)$ also admits a convex cocompact representation into $\mathsf{PO}(3,1)$. 
One might ask by extension:

\medskip\noindent
{\bf Question 6:} {\em Are there hyperbolic groups which admit discrete faithful linear representations, but do not
admit Anosov representations?}

\end{document}